\documentclass[a4paper,12pt,intlimits]{amsart}
\usepackage{amssymb}
\usepackage{graphicx}  \usepackage{epstopdf}

\usepackage[english]{babel}
\usepackage{amsmath}

\numberwithin{equation}{section}

\newtheorem{thm}{Theorem}[section]

\newtheorem{cor}[thm]{Corollary}
\newtheorem{prop}[thm]{Proposition}
\newtheorem{lem}[thm]{Lemma}

\newcommand{\R}{\mathbb{R}}
\newcommand{\N}{\mathbb{N}}
\newcommand{\Z}{\mathbb{Z}}

\newcommand{\al}{\alpha}

\newcommand{\xs}{\overline{x}}

\renewcommand{\theta}{\vartheta}
\renewcommand{\epsilon}{\varepsilon}
\newcommand{\ep}{\epsilon}
\newcommand{\I}{\mathcal{I}_s}

\newcommand{\cs}{\overline{c}}
\newcommand{\vs}{\overline{v}}
\newcommand{\xss}{\hat{x}}
\newcommand{\css}{\hat{c}}
\newcommand{\dss}{\hat{\delta}}

\newcommand{\thh}{\hat{\theta}}

\renewcommand{\leq}{\leqslant}
\renewcommand{\le}{\leqslant}
\renewcommand{\geq}{\geqslant}
\renewcommand{\ge}{\geqslant}

\newcommand{\beq}{\begin{equation}}
\newcommand{\eeq}{\end{equation}}
\newcommand{\beqs}{\begin{equation*}}
\newcommand{\eeqs}{\end{equation*}}
\newcommand{\beqa}{\begin{eqnarray}}
\newcommand{\eeqa}{\end{eqnarray}}
\newcommand{\beqas}{\begin{eqnarray*}}
\newcommand{\eeqas}{\end{eqnarray*}}

\title[Crystal dislocation dynamics]{Long-time behavior for crystal dislocation dynamics}

\author{Stefania Patrizi and Enrico Valdinoci}
\thanks{The authors have been supported by the
ERC grant 277749 ``EPSILON Elliptic
Pde's and Symmetry of Interfaces and Layers for Odd Nonlinearities''}

\address[Stefania Patrizi]{
Department Of Mathematics,
University of Texas at Austin,
2515 Speedway, Austin TX 78712, United States}

\address[Enrico Valdinoci]{
School of Mathematics and Statistics,
University of Melbourne,
813 Swanston Street, Parkville VIC 3010, Australia}

\address[Enrico Valdinoci]{
Universit\`a degli studi di Milano,
Dipartimento di Matematica,
Via Saldini 50, 20133 Milan, Italy}

\address[Stefania Patrizi and Enrico Valdinoci]{
Weierstra{\ss} Institut f{\"u}r Angewandte und Stochastik,
Mohrenstra{\ss}e 39, D-10117 Berlin, Germany}

\email{spatrizi@math.utexas.edu} 
\email{enrico@mat.uniroma3.it}

\begin{document}

\begin{abstract}
We describe the asymptotic states for the solutions of
a nonlocal equation of evolutionary type, which have the
physical meaning of the atom dislocation function in a periodic crystal.

More precisely, we can describe accurately the ``smoothing effect''
on the dislocation function occurring slightly after a
``particle collision'' (roughly speaking, two opposite transitions
layers average out) and, in this way, we can
trap the atom dislocation function between a superposition
of transition layers which, as time flows, approaches
either a constant function
or a single heteroclinic (depending on the algebraic properties
of the orientations of the initial transition layers).

The results are endowed of explicit and quantitative
estimates and, as a byproduct, we show that
the ODE systems of particles that governs the evolution
of the transition layers does not admit stationary solutions
(i.e., roughly speaking, transition layers always move).
\end{abstract}
\maketitle

\section{Introduction}

In the scientific literature, several
models have been considered in order to describe the motion
of the atom dislocations in a crystal. Roughly speaking,
a crystal is a structure in which the atoms have
the strong tendency to occupy some given site of a lattice;
nevertheless, some atom may occupy a different position that
the one at rest, and an
important question is
the accurate description of the evolution of this dislocation function
and of its
asymptotic and stationary behaviors.

Since different scales come into play in such description,
different models have been adopted, in order to
deal with phenomena at the atomic, microscopic, mesoscopic
and macroscopic scale. Goal of this paper is to consider
a microscopic model, inspired by 
(and, in fact, even more general than)
the classical one by
Peierls and Nabarro, see e.g.~\cite{nabarro}
for a detailed description
and also Section~2 in~\cite{dpv} for a simple
introduction.\medskip

In this setting, after a suitable section of a three-dimensional
crystal with a transverse plane,
the edge dislocation of the atoms along a slip plane
is described by a function~$v_\varepsilon=v_\varepsilon(t,x)$, where~$t\ge0$ is
the time variable, $x\in\R$ is the space variable
and~$\varepsilon>0$ is the characteristic length of the crystal
(say, roughly speaking, the distance between the minimal rest positions
of the crystal atoms).

The function~$v_\varepsilon$ satisfies a nonlocal equation since
the evolution along the slip plane is
influenced by the whole structure of
the crystal, which favors the rest position of the atoms
in a lattice, 
that, in our case, will be taken to be~$\Z$.

More precisely, the influence of the elastic energy of the whole
crystal along the slip plane produces a fractional operator,
which we denote by~$\I$ and which
is balanced by a force coming from a periodic multi-well
potential~$W$
produced by the periodic structure of the crystal in the large.

The presence of an external stress~$\sigma$ can also be taken into account
(of course, if one aims at ``general'' results, one has to
assume that this stress is sufficiently small
to allow a long-time behavior in which the structure
of the crystal is dominant with respect to the external forces).
\medskip

In further detail,
we consider here the initial value problem 
\beq\label{vepeq}\begin{cases}
\partial_t v_\ep=\displaystyle\frac{1}{\ep}\left(
\I v_\ep-\displaystyle\frac{1}{\ep^{2s}}W'(v_\ep)+\sigma(t,x)\right)&\text{in }(0,+\infty)\times\R\\
v_\ep(0,\cdot)=v_\ep^0&\text{on }\R
\end{cases}\eeq
where $\ep>0$ is a small scale parameter, $W$ is a periodic potential and  $\I$ is the so-called fractional Laplacian 
 of any order $2s\in(0,2)$,
that we define (up to a multiplicative normalization constant
that we neglect) as 
$$ \I[\varphi](x):=\frac12
\displaystyle\int_{\R}\displaystyle\frac{\varphi(x+y)+
\varphi(x-y)-2\varphi(x)}{|y|^{1+2s}}dy.$$
When~$s=\frac12$ and~$W(v):=1-\cos(2\pi v)$,
stationary solutions of~\eqref{vepeq} correspond to equilibria
in the classical model for dislocation dynamics of
Peierls and Nabarro~\cite{nabarro} 
(and indeed the results that we present are new even for
such model case). See also~\cite{s} or~\cite{dnpv} for a basic introduction
to the fractional Laplace operator. 

We assume that~$W$ is a multi-well potential with nondegenerate
minima at integer points.
More precisely, we suppose that
\begin{equation}\label{Wass}
\begin{cases}W\in C^{3,\alpha}(\R)& \text{for some }0<\alpha<1\\
W(v+1)=W(v)& \text{for any } v\in\R\\
W=0& \text{on }\Z\\
W>0 & \text{on }\R\setminus\Z\\
W''(0)>0.\\
\end{cases}
\end{equation}
The function $\sigma$ represents the external stress
and we assume on it the following regularity
conditions:
\begin{equation}\label{sigmaassump}
\begin{cases}
\sigma \in BUC([0,+\infty)\times\R)\quad\text{and for some }M>0\text{ and }\alpha\in(s,1)\\
\|\sigma_x\|_{L^\infty([0,+\infty)\times\R)}+\|\sigma_t\|_{L^\infty([0,+\infty)\times\R)}\leq M\\
|\sigma_x(t,x+h)-\sigma_x(t,x)|\leq M|h|^\alpha,\quad\text{for every }x,h\in\R \text{ and }t\in[0,+\infty).
\end{cases}
\end{equation}
\medskip

In order to detect the long-time evolution of the system
in~\eqref{vepeq}, we consider initial values that come
from a ``finite (but arbitrarily large)
number'' of single atom dislocations. 
%% Namely,
%% we assume that the initial condition in~\eqref{vepeq} is
%% a superposition of $N$ transition layers, $K$ of which
%% are positively oriented and~$N-K$ negatively oriented
%% (of course, $K$, $N\in\N$, and~$K\le N$).

To make this assumption more explicit,
we introduce the so-called basic layer solution $u$ associated to $\I$
(see~\cite{psv, cs, cozzi}), that is the solution of
the stationary equation
\begin{equation}\label{u}
\begin{cases}\I(u)=W'(u)&\text{in}\quad \R\\
u'>0&\text{in}\quad \R\\
\displaystyle\lim_{x\rightarrow-\infty}u(x)=0,\quad\displaystyle\lim_{x\rightarrow+\infty}u(x)=1,\quad u(0)=\displaystyle\frac{1}{2}.
\end{cases}
\end{equation}
Given $x_1^0<x_2^0<\dots<x_N^0$, 
 we say that the function $u\left(\frac{x-x_i^0}{\ep}\right)$ 
is a transition layer centered at $x_i^0$ and positively oriented. 
Similarly, we say that the function
$u\left(\frac{x_i^0-x}{\ep}\right)-1$
is a transition layer centered at $x_i^0$ and negatively oriented.

We observe that a positive oriented transition layer connects
the integer values~$0$ and~$1$, with a transition that
becomes steeper and steeper as~$\ep\to0$. Viceversa,
a negative oriented transition layer connects
the integer values~$0$ and~$-1$.

In this setting, we consider as initial condition in \eqref{vepeq}
the superposition of~$K$ positive oriented transition layers
with~$N-K$ negative oriented transition layers (modified
by a small term which takes into account the possible
reaction to an external stress), given by the formula
\beq\label{vep03}
v_\ep^0(x):=\displaystyle\frac{\ep^{2s}}{\beta}
\sigma(0,x)+ \sum_{i=1}^Nu\left(\zeta_i
\displaystyle\frac{x-x_i^0}{\ep}\right)-(N-K),
\eeq 
where $u$ is
solution of \eqref{u}, $\zeta_1,\dots,\zeta_N\in\{-1,1\}$, $\displaystyle\sum_{i=1}^N(\zeta_i)^+=K$, $0\leq K\leq N$ and
\beq\label{beta}\beta:=W''(0)>0.\eeq 
We observe that when~$\zeta_i=1$, the $i$th transition layer
in~\eqref{vep03} is positively oriented, while
when~$\zeta_i=-1$,
it is negatively oriented.
We also point out that, if~$\sigma\equiv0$, then
\begin{equation}\label{YU78hG12SA}
\begin{split}
& \lim_{x\to-\infty} v_\ep^0(x)=
\sum_{{1\le i\le N}\atop{\zeta_i=-1}} 1 - (N-K) =0\\
{\mbox{and }}
& \lim_{x\to+\infty} v_\ep^0(x)=
\sum_{{1\le i\le N}\atop{\zeta_i=1}} 1 - (N-K) = 2K-N.\end{split}\end{equation}
\medskip

It has been shown in~\cite{gonzalezmonneau} (when~$s=\frac12$),
in~\cite{dpv} (when~$s\in\left(\frac12,1\right)$)
and in~\cite{dfv} (when~$s\in\left(0,\frac12\right)$)
that the evolution of~$v_\ep$ with the initial condition
in~\eqref{vep03} resembles, as~$\ep\to0$, a step functions
with integer values, whose $N$ points of discontinuity,
say~$(x_1(t),\ldots,x_N(t))$,
move according to a dynamical system.
More precisely, as proved in~\cite{pv2},
the potential that drives this dynamical system is either
repulsive (when the associated transition layers have the same
orientations) or attractive (when they have opposite orientations).
In case of attractive potentials, these discontinuity points
(sometimes referred in a suggestive but perhaps
a bit improper way with the name of ``particles'')
collide in a finite time~$T_c$, see again~\cite{pv2}
for a detailed description of this phenomenon.
\medskip

The explicit system of ordinary differential equations
which govern the motion of these jump points~$
(x_1(t),\ldots,x_N(t))$ is given by
\beq\label{dynamicalsysNeven}\begin{cases}\dot{x}_i=\gamma\left(
\displaystyle\sum_{j\neq i}\zeta_i\zeta_j 
\displaystyle\frac{x_i-x_j}{2s |x_i-x_j|^{1+2s}}-\zeta_i\sigma(t,x_i)\right)&\text{in }(0,T_c)\\
 x_i(0)=x_i^0,
\end{cases}\eeq
for $i=1,\ldots,N$,
where  \beq\label{gamma}\gamma:=\left(\displaystyle\int_\R (u'(x))^2 dx\right)^{-1},\eeq and $0<T_c\leq+\infty$ is the collision time 
of system \eqref{dynamicalsysNeven}.

More explicitly, a collision time~$T_c$ is characterized by the fact that
 $$x_{i+1}(t)>x_{i}(t)\quad\text{for any }t\in[0,T_c)\text{ and }i=1,\dots, N-1$$ and there exists $i_0$ such that 
 $$x_{i_0+1}(T_c)=x_{i_0}(T_c).$$
\medskip

If a collision occurs, after the collision time~$T_c$, the dynamical system
in~\eqref{dynamicalsysNeven} (as given in~\cite{gonzalezmonneau, dpv, dfv, pv2})
ceases to be well-defined,
since at least one of the denominators vanishes, hence
the mesoscopic description in the
limit as~$\ep\to0$ ceases to be available. Nevertheless,
for a fixed~$\ep>0$, the solution~$v_\ep$ of
the evolution equation~\eqref{vepeq} continues to exist
and to describe the dislocation dynamics.\medskip 

In~\cite{pv3}, we gave
a first explicit description of what happens to the solution~$v_\ep$
after the collision time when {\em only two or three}
layer solutions are taken into account.
Goal of this paper is to further extend this study, by taking into
account the superposition of {\em any number of transition layers},
by describing qualitatively the {\em asymptotic states} and by
providing {\em quantitative estimates} on
the relaxation times needed to approach the limits.\medskip

To this goal, we consider several cases, such as: 
\begin{itemize}
\item the situation in which 
the first $K$ transition layers are positively oriented and the 
remaining last $N-K$ negatively oriented (we call this situation
the ``segregate orientation'' case),
\item the situation in which there are as many positively oriented
as negatively oriented transition layers (we call this situation
the ``balanced orientation'' case),
\item the situation in which there are more positively oriented
than negatively oriented transition layers (we call this situation
the ``unbalanced orientation'' case;
of course the opposite
situation in which there are more negatively oriented
than positively oriented transition layers can be reduced to this case,
up to a spacial reflection).
\end{itemize}

The results that we obtain are naturally different according to
the different cases. In the segregate orientation case
we will show that, roughly speaking, 
the last ``positively oriented particle''
in the dynamical system~\eqref{dynamicalsysNeven} will collide
with the first ``negatively oriented particle''
at some time~$T_c$; then, slightly after~$T_c$, two
transition layers of the solution~$v_\ep$ will merge the one into
the other and annihilate each other (as a consequence,
after this, the solution~$v_\ep$ somehow decreases
its oscillations).

We remark that the segregate orientation case
is not only interesting in itself, but it also provides a natural
comparison for the general case (i.e. it provides the necessary
barriers for the other cases, thus reducing each time the picture
to the ``worst possible scenario'').\medskip

The balanced orientation case presents the special feature of
having~$K=N-K$, that is~$N=2K$, which says that the dislocation
function goes to zero both at~$-\infty$ and at~$+\infty$
(recall~\eqref{YU78hG12SA}). These conditions at infinity
influence the asymptotic behavior in time of~$v_\ep$, since
we will show that,
after a transient time in which collisions occur,
the solution~$v_\ep$ relaxes to zero
exponentially fast.\medskip

The unbalanced orientation case is somehow more complex.
In this case, we have~$K > N-K$, so we set~$l:=
2K-N=K-(N-K)>0$ (notice that~$l$ is the difference between
positively oriented and negatively oriented initial transitions).
In this situation, the initial dislocation approaches zero at~$-\infty$
and~$l$ as~$x\to+\infty$ (recall again~\eqref{YU78hG12SA}).

The asymptotics in time of the dislocation function~$v_\ep$
is again influenced by these conditions at infinity, since, roughly
speaking, the limit behavior as~$t\to+\infty$ will try to make an
average between the two values at infinity.
On the other hand, this ``exact'' average procedure
is not (always) possible for the system
and indeed it is not (always) true that~$v_\ep$
approaches the constant value~$\frac{l}{2}$ as~$t\to+\infty$.

The heuristic reason for this fact is that the constant~$\frac{l}{2}$
is not necessarily a solution of the stationary equation,
and even when it is a solution (as in the model case given by
the choice of the potential~$W(v):=1-\cos(2\pi v)$) such solution
is unstable from the variational point of view.

In fact, we will show that the constant value~$\frac{l}{2}$ 
is only reached as~$t\to+\infty$
``in average'' in a possibly dynamical way
and in a way which is compatible
with the stable solutions of the stationary equation.
Namely, if~$\frac{l}{2}\in\N$ (i.e. $l$ is even)
then\footnote{It is worth to point out that, as expected,
the unbalanced orientation case
boils down to the balanced orientation case when~$l=0$ (in any case,
the quantitative estimates that we obtain in the balanced case
are more explicit and precise than the ones for the unbalanced case).}
indeed~$v_\ep\to \frac{l}{2}$ as~$t\to+\infty$;
but if instead~$\frac{l}{2}\not\in\N$ (i.e. $l$ is odd) then, for large times,
the dislocation function~$v_\ep$ will approach a transition layer
which joins the integer~$(l-1)/2$ at~$-\infty$
with the integer~$(l+1)/2$ at~$+\infty$ (that is,
a vertical translation of the standard heteroclinic from~$0$ to~$1$).
Thus, when~$l$ is odd, the constant value~$\frac{l}{2}$
is not attained in the limit~$t\to+\infty$, but instead
the system attains a dynamic connection between the values~$\frac{l}{2}-
\frac{1}{2}$ and~$\frac{l}{2}+\frac{1}{2}$.
\medskip

All these statements will be proved in a quantitative way,
by using appropriate comparison functions.
We now give a precise mathematical statements of the results
that we have just described in words.

\subsection{The segregate orientation case}

We first consider the particular case in which 
the first $K$ transition layers in \eqref{vep03}  
are positively oriented and the remaining last $N-K$ negatively oriented, i.e., we assume 
\begin{equation}\label{kposin-knegzeta}\zeta_i=\begin{cases}1&\text{for }i=1,\ldots,K\\
-1&\text{for }i=K+1,\ldots,N.
\end{cases}
\end{equation}
Under this\footnote{As a matter of fact,
we will show in Lemma \ref{coo1}, that
if $1-2s\theta_0^{2s}\|\sigma\|_\infty>0$
and~\eqref{kposin-knegzeta} holds true,
then a collision always occurs in a finite time, i.e.,  $T_c<+\infty$.} 
assumption, we show that if the collision time $T_c$ is finite, then the collision occurs between  particles $x_K$ and  $x_{K+1}$, and after a time $T_\ep$, which is slightly larger than $T_c$, the function $v_\ep$ is dominated by the superposition of $N-2$ transition layers, the first $K-1$ of them positively oriented 
and the last $N-K-1$ negatively oriented. 

The precise mathematical statement goes as follows:

\begin{thm}\label{mainthmbeforecollNeven}
Assume that  \eqref{Wass}, \eqref{sigmaassump}, \eqref{kposin-knegzeta} hold, that~$0<K<N$ and that~$T_c<+\infty$. Let $v_\ep$ be the solution of \eqref{vepeq}-\eqref{vep03} and $(x_1(t),\ldots,x_N(t))$ the solution of \eqref{dynamicalsysNeven}.
Then there exist $\ep_0>0$  and $c>0$ such that for any $\ep<\ep_0$   there exist 
$x_1^\ep,\ldots,x_{K-1}^\ep,x_{K+2}^\ep,\ldots, x_{N}^\ep\in\R$ and $T_\ep,\,\varrho_\ep>0$, such that for $i\in\{1,\ldots,K-1,K+2,\ldots N\}$,  
\beq\label{xiodpertumainthm}x_i^\ep=x_i(T_c)+o(1)\quad \text{ as }\ep\to0,%\begin{cases} x_i(T_c)+o(1)&\text{for }i=1,\ldots,K-1\\
%x_{i+2}(T_c)+o(1)&\text{for }i=K,\ldots,N-2,
%\end{cases} 
\eeq
\beq\label{xi+1-xiboundmainthm1} x_{i+1}^\ep-x_i^\ep\geq c,\eeq 
 $$T_\ep=T_c+o(1)\quad\text{as } \ep\to 0,$$
 \beq\label{varroepmainthm}\varrho_\ep=o(1),\quad \frac{\ep^{2s}}{\varrho_\ep}=o(1),\quad \frac{\varrho_\ep}{\ep^s}=o(1)\quad\text{as } \ep\to 0
 \eeq and for any $x\in\R$,

 \beq\label{vlesepN-2} v_\ep(T_\ep,x)\leq \displaystyle\frac{\ep^{2s}}{\beta}
\sigma(T_\ep,x)+ \sum_{i=1}^{K-1}u\left(
\displaystyle\frac{x-x_i^\ep}{\ep}\right)+ \sum_{i=K+2}^{N}u\left(
\displaystyle\frac{x_i^\ep-x}{\ep}\right)-(N-K-1)+ \varrho_\ep,\eeq
where $u$ is the solution of \eqref{u} and $\beta$ is given by \eqref{beta}.
\end{thm}

The evolution of the dislocation function~$v_\ep$ from~$t<T_c$ to~$t>T_c$
is described in Figure~1 (roughly speaking, right after the collision
of the $K$th particle with the $(K+1)$th particle, the dislocation
averages out one oscillation).

\bigskip

\begin{center}
  \includegraphics[width=1.01\textwidth]{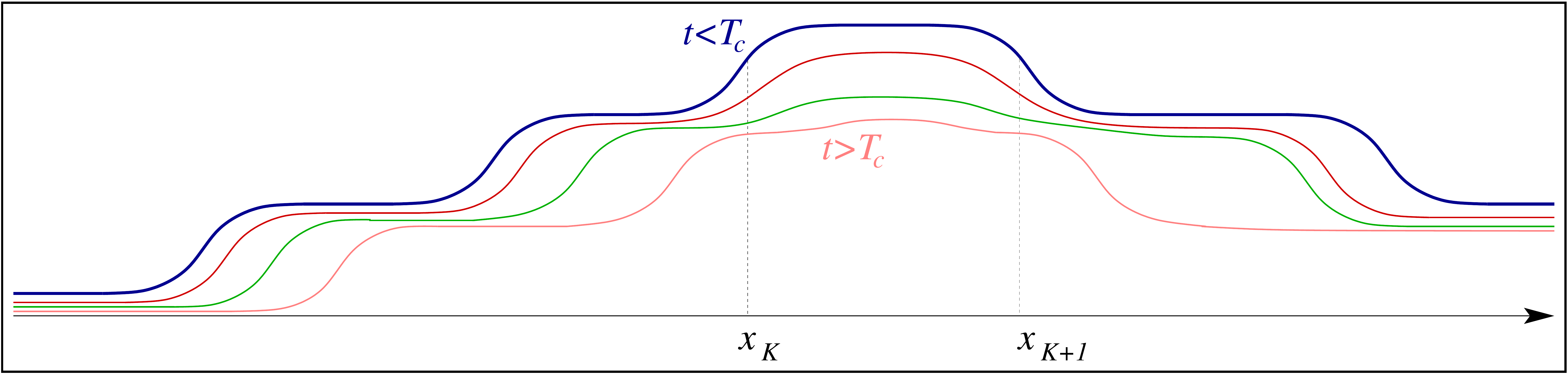}
{\footnotesize{
\nopagebreak$\,$\\
\nopagebreak
Figure 1 (segregate orientation case): Evolution of the 
dislocation function as described in 
Theorem~\ref{mainthmbeforecollNeven}.}}
\end{center}

\bigskip
\bigskip

In addition, we can better
quantify Theorem~\ref{mainthmbeforecollNeven}. Indeed,
the error term $\varrho_\ep$ in \eqref{vlesepN-2} becomes smaller than $\ep^{2s}$ after an additional small time $\tau_\ep$ as shown in the next theorem, as stated below.

\begin{thm}\label{thmexponentialdecayNeven} Under the assumptions of Theorem  \ref{mainthmbeforecollNeven}, if $N>2$, then there exists $\ep_0>0$ such that for any $\ep<\ep_0$ there 
 there exist 
$\widetilde{x}_1^\ep,\ldots,\widetilde{x}_{K-1}^\ep,\widetilde{x}_{K+2}^\ep,\ldots,\widetilde{x}_{N}^\ep\in\R$, and $\widetilde{\varrho}_\ep,\,\tau_\ep>0$, such that
\beq\label{tauepvarromexpthm}\tau_\ep=o(1),\quad\widetilde{\varrho}_\ep=o(1)\ep^{2s} \quad\text{as }\ep\to 0,\eeq for $i\in\{1,\ldots,K-1,K+2,\ldots N\}$,
\beq\label{xi-tildexiexpthm}|\widetilde{x}_i^\ep-x_i^\ep|=o(1)\quad\text{as }\ep\to 0,\eeq and 
 \beq\label{vlesepN-2exponthm} v_\ep(T_\ep+\tau_\ep,x)\leq \displaystyle\frac{\ep^{2s}}{\beta}
\sigma(T_\ep+\tau_\ep,x)+ \sum_{i=1}^{K-1}u\left(
\displaystyle\frac{x-\widetilde{x}_i^\ep}{\ep}\right)+ \sum_{i=K+2}^{N}u\left(
\displaystyle\frac{\widetilde{x}_i^\ep-x}{\ep}\right)-(N-K-1)+\widetilde{\varrho}_\ep,\eeq
where $T_\ep$ and the $x_i^\ep$'s are given in Theorem  \ref{mainthmbeforecollNeven},   $u$ is the solution of \eqref{u} and $\beta$ is given by \eqref{beta}.
\end{thm}

\subsection{The balanced orientation case}

Now we consider the case in which $K=N-K$, i.e.
the initial configuration presents as many positively oriented layers
as negatively oriented ones. In this case, we will use
Theorem \ref{thmexponentialdecayNeven} to construct a barrier for the
evolution of~$v_\ep$. Namely,
by an appropriate iteration of Theorem \ref{thmexponentialdecayNeven},
we show that, given any initial configuration of an equal number of
positive and negative initial dislocations, the system
relaxes to the trivial equilibrium (and the relaxation
times are exponential). The precise results are stated as follows:

\begin{thm}\label{Ncollisionevencor}
Assume that  \eqref{Wass}, \eqref{sigmaassump},  hold and that $$N=2K.$$ 
Let $v_\ep$ be the solution of \eqref{vepeq}-\eqref{vep03}.
Then there exist  $\overline{\sigma}>0$ and $\ep_0>0$, such that 
if 
\beq\label{sigmaNcollcond}\|\sigma\|_{\infty}\leq \bar{\sigma},\eeq then  
for any $\ep<\ep_0$ and any $(\zeta_1,\ldots,\zeta_N)\in\{-1,1\}^N$ such that $\sum_{i=1}^N\zeta_i=0$, there exist $\mathcal{T}_\ep^K,\Lambda_\ep^K>0$ such that 
\beq\label{vepNcoll}|v_\ep(\mathcal{T}_\ep^K,x)|\leq \Lambda_\ep^K,\quad\text{for any }x\in \R,\eeq and
\beq\label{lambdaK}\Lambda_\ep^K=o(1)\quad\text{as }\ep\to0.\eeq
\end{thm}

\begin{thm}\label{Ncollisionevencorsigma0} Under the assumptions of Theorem~\ref{Ncollisionevencor}, if in addition $\sigma\equiv 0$, then there exist $\ep_0>0$ and $c>0$ such that for any $\ep<\ep_0$ we have 
\beq\label{vexpontozeroNeven}|v_\ep(t,x)|\leq \Lambda_\ep^K e^{c\frac{\mathcal{T}_\ep^K-t}{\ep^{2s+1}}},\quad\text{for any }x\in\R\text{ and }t\geq \mathcal{T}_\ep^K,\eeq where 
$\mathcal{T}_\ep^K$ 
and $\Lambda_\ep^K$ are given in Theorem~\ref{Ncollisionevencor}. 
 \end{thm}

We observe that the exponential decay (for large~$t$)
given in~\eqref{vexpontozeroNeven}
becomes stronger and stronger for small values of the positive parameter~$\ep$
(i.e. a small scale of the crystal favors the relaxation of
the system).\medskip

The situation analytically described in Theorems~\ref{Ncollisionevencor}
and~\ref{Ncollisionevencorsigma0} is depicted in Figure~2.
\bigskip

\begin{center}
  \includegraphics[width=1.01\textwidth]{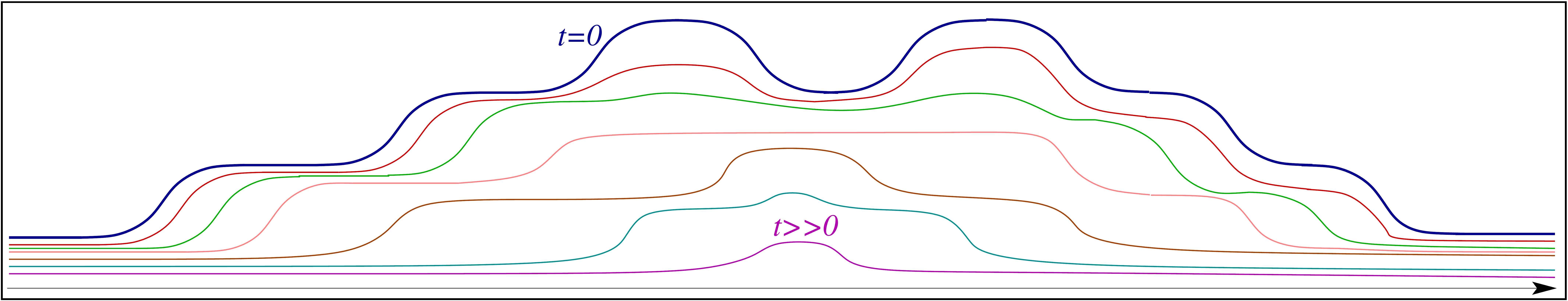}
{\footnotesize{
\nopagebreak$\,$\\
\nopagebreak
Figure 2 (balanced orientation case): Evolution of the dislocation function as described in
Theorems~\ref{Ncollisionevencor}
and~\ref{Ncollisionevencorsigma0}.}}
\end{center}
\bigskip
\bigskip

It is worth to point out that
the threshold~$\bar{\sigma}$ in~\eqref{sigmaNcollcond}
is obtained here by the method of continuity from the case~$\sigma\equiv0$;
of course, also in view of concrete applications, we think that
it is an interesting problem to obtain explicit quantitative
bounds on~$\bar{\sigma}$.
 
\subsection{The unbalanced orientation case}

Now we turn to the general case in which
the number of positive initial orientations is not necessarily the same
as the number of negative ones. In this case,
the limit configuration is either a constant or
a single transition, according to the parity of the difference
between positive and negative initial orientations.
The precise statements go as follows:

 \begin{thm}\label{NcollisionN=2K+lcor}
Assume that  \eqref{Wass}, \eqref{sigmaassump},  \eqref{vep03} hold and that $$N=2K-l,\quad l\in \N.$$
Let $v_\ep$ be the solution of \eqref{vepeq}-\eqref{vep03}. 
Then there exist~$\overline{\sigma}>0$ and $\ep_0>0$ such that 
if 
\beq\label{sigmaNcollN=2K+lcorcond}\|\sigma\|_{\infty}\leq \bar{\sigma},\eeq 
for any $\ep<\ep_0$ and any $(\zeta_1,\ldots,\zeta_N)\in\{-1,1\}^N$ such that $\sum_{i=1}^N\zeta_i=l$,  
there exist $\mathcal{T}_\ep^{K-l},\Lambda_\ep^{K-l}>0$, $\overline{x}_1^\ep,\ldots, \overline{x}_l^\ep$, $\underline{x}_1^\ep,\ldots, \underline{x}_l^\ep\in\R$, bounded with respect to $\epsilon$, with $\underline{x}_i^\ep\leq \overline{x}_i^\ep$, 
 such that  for any $x\in\R$ 
\beq\label{vepNcollN=2K+lcorabove}v_\ep(\mathcal{T}_\ep^{K-l},x)
\leq \displaystyle\frac{\ep^{2s}}{\beta}
\sigma(\mathcal{T}_\ep^{K-l},x)+ \sum_{i=1}^{l}u\left(
\displaystyle\frac{x-\underline{x}_i^\ep}{\ep}\right)+\Lambda_\ep^{K-l},\eeq 
and 
\beq\label{vepNcollN=2K+lcorbelow}v_\ep(\mathcal{T}_\ep^{K-l},x)\geq \displaystyle\frac{\ep^{2s}}{\beta}
\sigma(\mathcal{T}_\ep^{K-l},x)+ \sum_{i=1}^{l}u\left(
\displaystyle\frac{x-\overline{x}_i^\ep}{\ep}\right)-\Lambda_\ep^{K-l},\eeq 
where 
\beq\label{lambdak-lcor}\Lambda_\ep^{K-l}=o(\ep^{2s})\quad\text{as }\ep\to 0,\eeq
$u$ is the solution of \eqref{u} and $\beta$ is given by \eqref{beta}.
\end{thm}

\begin{thm}\label{NcollisionN=2K+lcorinfinity}
Under the assumptions of Theorem~\ref{NcollisionN=2K+lcor}, if in addition $\sigma\equiv 0$, then there exists $\ep_0>0$ such that  for any $\ep<\ep_0$, we have:
for any
$R>0$ there exists $T_0>\mathcal{T}_\ep^{K-l}$ such that
  for any $|x|\le R$ and $t>T_0$,
\begin{itemize}
\item if $l=2m$, $m\in\N$, then
 \beq\label{vinfinitycorleven}- C\ep^{2s}(1+t)^{-\frac{2s}{2s+1}}\le v_\ep(t,x)-m\le  C\ep^{2s}(1+t)^{-\frac{2s}{2s+1}}.\eeq 
% In particular,  $\lim_{t\to+\infty}v_\ep(t,x)=m,$ for any $x\in\R$.
 
\item  If  $l=2m+1$, $m\in\N$, then
  \beq\label{vinfinitycorlodd1}\begin{split} v_\ep(t,x)\ge m+u\left(\displaystyle\frac{x-\overline{x}^\ep-\alpha_\ep[(1+t)^\frac{1}{1+2s}-1]}{\ep}\right)-C \ep^{2s}(1+t)^{-\frac{2s}{2s+1}},
   \end{split}\eeq
   and 
     \beq\label{vinfinitycorlodd2}v_\ep(t,x)\le m+u\left(\displaystyle\frac{x-\underline{x}^\ep+\alpha_\ep[(1+t)^\frac{1}{1+2s}-1]}{\ep}\right)+C\ep^{2s} (1+t)^{-\frac{2s}{2s+1}},\eeq
 %$t_k\to+\infty$ as $k\to +\infty$,   and   a point $x_\ep\in\R$,  such that for any $x\in\R$, we have that
 %\beq\label{vinfinitycorlodd}v_\ep(t_k,x)\to u\left(\frac{x-x_\ep}{\ep}\right)+m\quad\text{as }k\to +\infty,\eeq
where  $u$ is the solution of \eqref{u}, $\alpha_\ep=o(1)$ as $\ep\to0$, 
 $\underline{x}^\ep,\,\overline{x}^\ep\in\R$ are bounded with respect to $\epsilon$ and  $\underline{x}^\ep\le\overline{x}^\ep$.
\end{itemize}
\end{thm}

\bigskip

\begin{center}
  \includegraphics[width=0.79\textwidth]{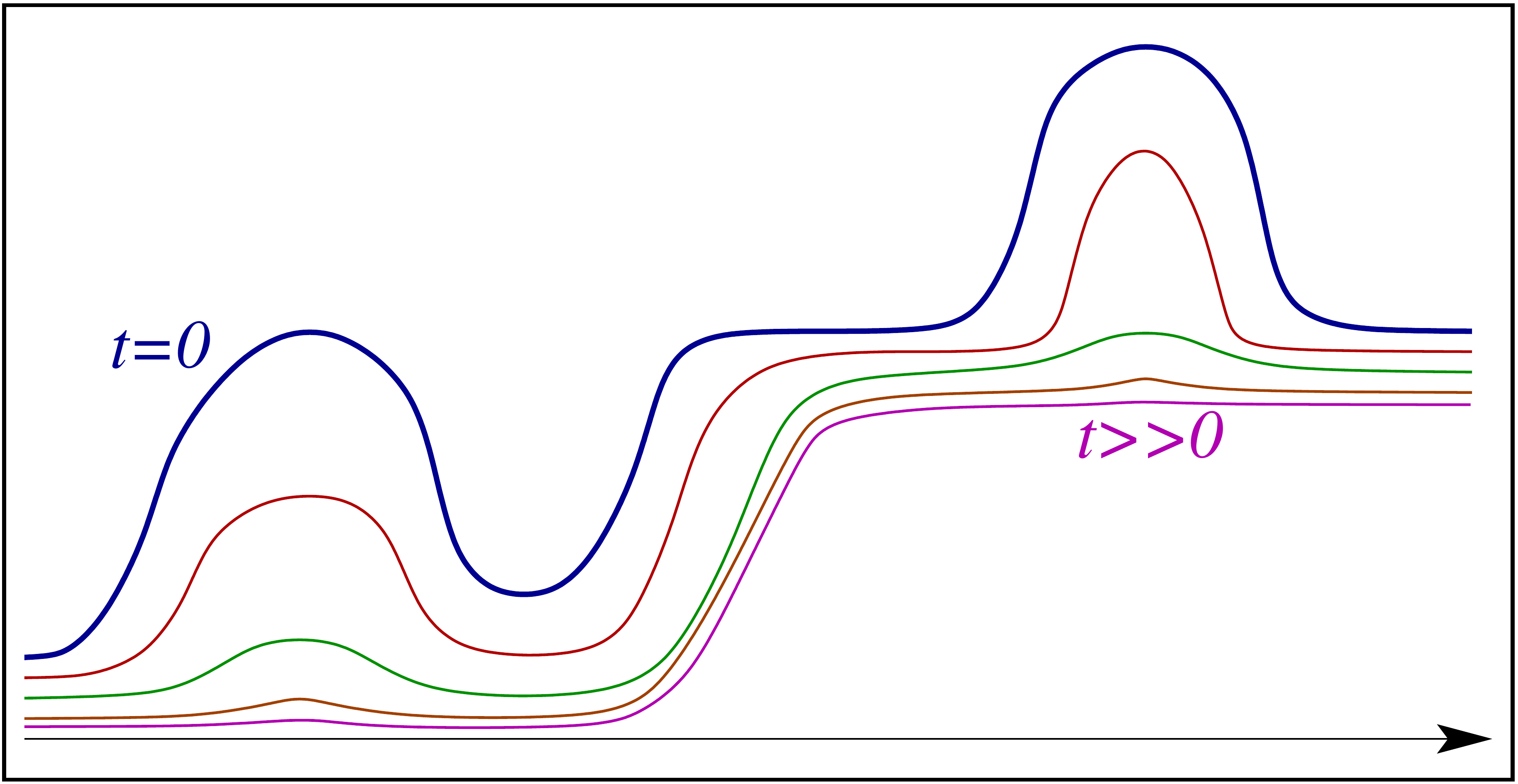}
{\footnotesize{\nopagebreak
$\,$\\
\nopagebreak Figure 3 (unbalanced orientation case): Evolution of
the dislocation function as described in
Theorems~\ref{NcollisionN=2K+lcor}
and~\ref{NcollisionN=2K+lcorinfinity} ($l$ odd, limit case: single 
transition).}}
\end{center}
\bigskip
\bigskip

The unbalanced case in which the dislocation function approaches
a single heteroclinic is depicted in Figures~3 and~4.
We also remark that the index~${K-l}$ in
Theorem~\ref{NcollisionN=2K+lcor} is related to the number
of iterations of Theorems~\ref{mainthmbeforecollNeven}
and~\ref{thmexponentialdecayNeven}
needed to perform its proof.
\medskip

In addition, we point out that there are some quantitative differences between
Theorem~\ref{Ncollisionevencorsigma0} and Theorem~\ref{NcollisionN=2K+lcorinfinity},
that is between the balanced and unbalanced orientation cases.

Indeed, when~$N=2K$ (i.e. $m=0$), the system relaxes
to zero exponentially fast, as given by~\eqref{vexpontozeroNeven}.
Conversely, when~$N\ne2K$, the relaxation times 
given in~\eqref{vinfinitycorleven},
\eqref{vinfinitycorlodd1} and~\eqref{vinfinitycorlodd2}
are only polynomial, due to the terms of order~$\ep^{2s} t^{-\frac{2s}{2s+1}}$
appearing in these formulas.

The fact is that, in the unbalanced orientation case,
the central points of the heteroclinics
which provide the barriers move and drifts to infinity:
for instance, in case~$m=1$, $N=K=2$, i.e. when two dislocations
with positive orientations are considered, the ODE system can
be solved explicitly and one sees that the distance between
the dislocations is of the order of~$t^{\frac{1}{1+2s}}$
(and this explains the term~$t^{\frac{1}{1+2s}}$ in the right hand sides
of~\eqref{vinfinitycorlodd1} and~\eqref{vinfinitycorlodd2}).

This quantitative remark also explains why the decay in time
in Theorem~\ref{NcollisionN=2K+lcorinfinity}
is polynomial (instead of exponential,
as it happens in Theorem~\ref{Ncollisionevencorsigma0}): indeed, the
heteroclinics mentioned above, which are centered
at distance~$O(t^{\frac{1}{1+2s}})$, possess a polynomial tail
(with power~$-2s$, see e.g. formula~(1.6)
in~\cite{dfv}): the (rescaled) combination of these two effects
produce an error of the form~$ \big( t^{\frac{1}{1+2s}}/\ep \big)^{-2s}$,
and this explains the term of order~$\ep^{2s} t^{-\frac{2s}{2s+1}}$
in~\eqref{vinfinitycorleven},
\eqref{vinfinitycorlodd1} and~\eqref{vinfinitycorlodd2}.
\bigskip

\begin{center}
  \includegraphics[width=0.79\textwidth]{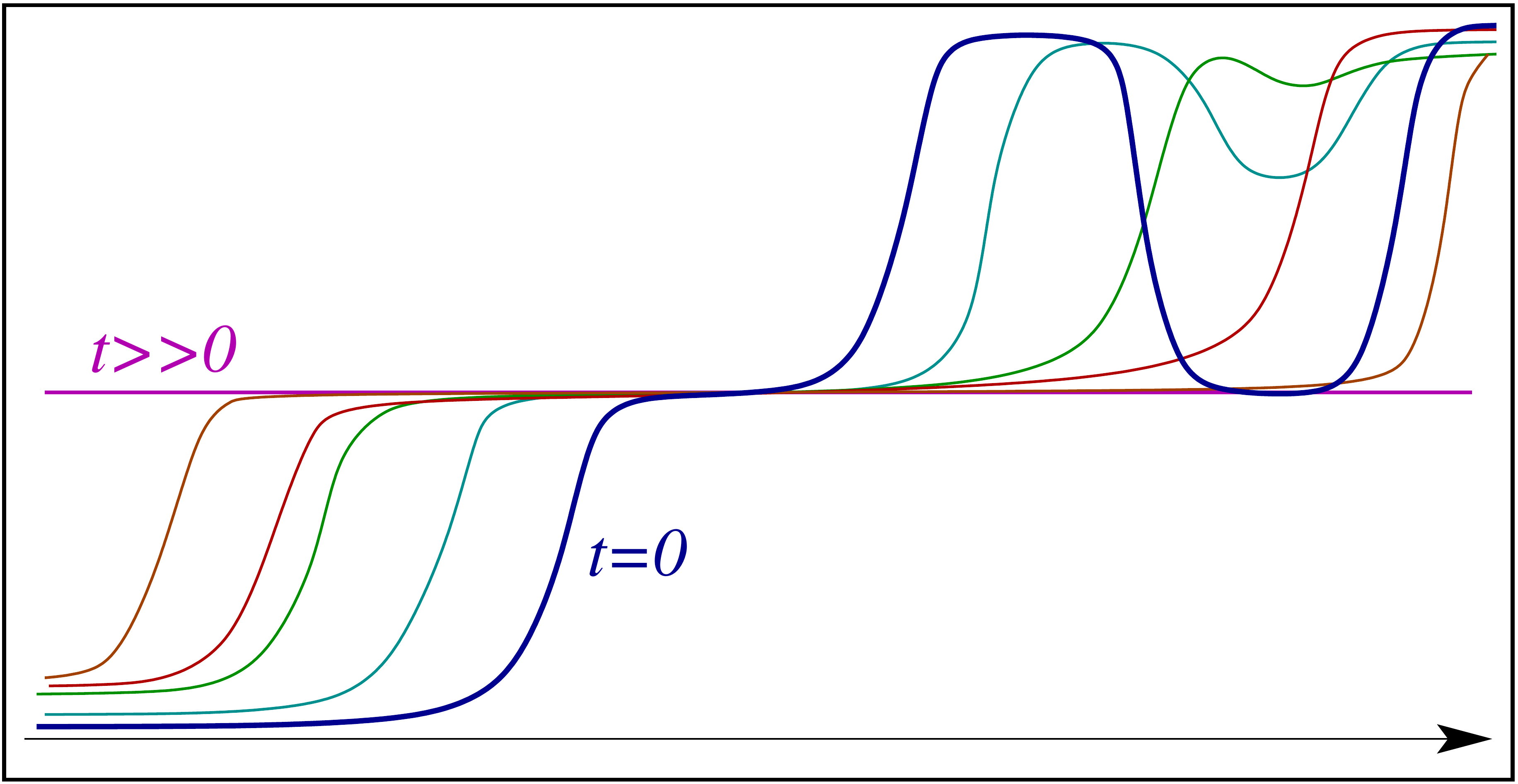}
{\footnotesize{\nopagebreak
$\,$\\
\nopagebreak Figure 4 (unbalanced orientation case): Evolution of 
the dislocation function as described in
Theorems~\ref{NcollisionN=2K+lcor}
and~\ref{NcollisionN=2K+lcorinfinity} ($l$ even, limit case:
constant).}}
\end{center}
\bigskip
\bigskip

\subsection{Equilibria of the dynamical system}

An interesting byproduct of our results
is that the particles in~\eqref{dynamicalsysNeven}
can never remain at rest, namely:

\begin{cor}\label{nostationarypoints}Assume that  \eqref{Wass}
holds true, that~$N\ge2$
and that~$\sigma\equiv 0$. Then the ODE system 
in~\eqref{dynamicalsysNeven} does not admit stationary points.
\end{cor}

It is worth to point out that a similar result does not hold for
infinitely many particles (an equilibrium being given
by alternate particles at the same distance).
It is also interesting to observe that our proof
of Corollary~\ref{nostationarypoints} is not based on ODE
methods, but on the analysis of the
integro-differential equation in~\eqref{vepeq},
which provides a further example of link between related,
but in principle different, topics, in terms of results,
motivations and methods.\medskip

The rest of the paper is organized as follows.
In Section~\ref{PP}
we collect a series of ancillary results, to be freely exploited
in the proofs of the main results.

Then, we prove
Theorem~\ref{mainthmbeforecollNeven} in Section~\ref{IAL9},
Theorem~\ref{thmexponentialdecayNeven} in Section~\ref{R F 24},
Theorems~\ref{Ncollisionevencor} 
and \ref{NcollisionN=2K+lcor} in Section~\ref{8jAJJa}, and
Theorems~\ref{Ncollisionevencorsigma0} 
and \ref{NcollisionN=2K+lcorinfinity} in Section~\ref{AKAaKA}.
Finally, Corollary~\ref{nostationarypoints} is proved in Section~\ref{FIF}.

\section{Preliminary observations}\label{PP}

\subsection{Toolbox}
In this section we recall some general auxiliary results that  will be used in the rest of the paper. In what follows we denote by $H$ the Heaviside function. 
 \begin{lem}\label{uinfinitylem} Assume that  \eqref{Wass} holds, then there exists a unique solution $u\in C^{2,\alpha}(\R)$ of \eqref{u}. Moreover,
  there exist   constants $C,c >0$ and $\kappa>2s$ (only depending on~$s$)
such that 
\begin{equation}\label{uinfinity}\left|u(x)-H(x)+\displaystyle\frac{1}{2sW''(0) 
}\displaystyle\frac{x}{|x|^{2s+1}}\right|\leq \displaystyle\frac{C}{|x|^{\kappa}},\quad\text{for }|x|\geq 1,%\quad\text{if }s\ge\displaystyle\frac{1}{2},
\end{equation} 
and
\begin{equation}\label{u'infinity}\frac{c}{|x|^{1+2s}}\leq u'(x)\leq \displaystyle\frac{C}{|x|^{1+2s}}\quad\text{for }|x|\geq 1.
\end{equation} 
\end{lem}
\begin{proof} The existence of a unique solution of \eqref{u} is proven in \cite{psv}, see also \cite{cs}. Estimate \eqref{uinfinitylem} is proven in \cite{gonzalezmonneau} for $s=\frac{1}{2}$ and in \cite{dpv}, \cite{dfv} respectively  for $s\in\left(\frac{1}{2},1\right)$ and  $s\in\left(0,\frac{1}{2}\right)$. Finally, estimate \eqref{u'infinity} is shown in \cite{cs}.
\end{proof}

Next, we introduce the function $\psi$ to be the solution of
\beq\label{psi}\begin{cases}
\I\psi-W''(u)\psi=u'+\eta(W''(u)-W''(0))&\text{in }\R\\
\psi(-\infty)=0=\psi(+\infty),
\end{cases}\eeq where $u$ is the solution of \eqref{u} and
\beq\label{eta}\eta:=\displaystyle\frac{1}{W''(0)}
\displaystyle\int_\R(u'(x))^2dx
=\displaystyle\frac{1}{\gamma \beta}.\eeq
For a detailed heuristic motivation of
equation \eqref{psi},
see Section 3.1 of \cite{gonzalezmonneau}.
For later purposes, we recall the following decay estimate
on the solution of~\eqref{psi}:

\begin{lem} Assume that  \eqref{Wass} holds, then there exists a unique solution $\psi$ to \eqref{psi}. Furthermore 
$\psi\in C^{1,\alpha}_{loc}(\R)\cap L^\infty(\R)$ for some $\al\in(0,1)$ and there exists $C>0$ such that for any $x\in\R$
\beq\label{psi'infty}|\psi'(x)|\le \frac{C}{1+|x|^{1+2s}}.\eeq
\end{lem}
\begin{proof}
The existence of a unique solution of  \eqref{psi} is proven in  \cite{gonzalezmonneau} for $s=\frac{1}{2}$  and  in \cite{dpv},  
\cite{dfv} respectively for $s\in\left(\frac{1}{2},1\right)$ and $s\in\left(0,\frac{1}{2}\right)$. Estimate \eqref{psi'infty} is shown in \cite{pv}.
\end{proof}

%%%%%%%%%%%%%%%%%%%%%%%%%%%%%%%%
%%%%%%%%%%%%%%%%%%%%%%%%%%%%%%%
\section{Proof of Theorem \ref{mainthmbeforecollNeven}}\label{IAL9}

 Let $(x_1(t),\ldots,x_N(t))$ be  the solution of \eqref{dynamicalsysNeven}, where the $\zeta_i$'s are given by \eqref{kposin-knegzeta}. 
Let us denote, for $i=1,\ldots, N-1$
\beqs \theta_i(t):=x_{i+1}(t)-x_{i}(t),\eeqs and 
\beqs \theta_i^0:=x_{i+1}^0-x_{i}^0.\eeqs
Let us start by showing that if the assumption \eqref{coo1} below is satisfied, then the condition $T_c<+\infty$ holds true  and a collision occurs between the particles $x_K$ and $x_{K+1}$. 
 \begin{lem}\label{T^1_cboundlemma} %Let $(x_1,\ldots,x_N)$ be the solution of \eqref{dynamicalsysNeven} 
 %with $\zeta_i=1$ for $i=1,\ldots, K$ and $\zeta_i=-1$ for $i=K+1,\ldots, N$.  
 Assume that 
 \beq\label{coo1}1-2s(\theta_{K}^0)^{2s}\|\sigma\|_\infty>0.\eeq Then 
 $\theta_{K}(t)$ is decreasing and there exists $T_c$ satisfying 
 \beq\label{T^1_cbound}T_c\leq\frac{s(\theta_{K}^0)^{1+2s}}{(2s+1)\gamma(1-2s(\theta_{K}^0)^{2s}\|\sigma\|_\infty)} ,\eeq
 such that 
 $$\theta_{K}(T_c)=0.$$
 \end{lem}

\begin{proof}
%Let us denote $\theta_K(t):=x_{K+1}(t)-x_{K}(t).$
From  \eqref{dynamicalsysNeven} and \eqref{kposin-knegzeta}, we infer that
\beqs\begin{split}
\dot{\theta}_K&=\dot{x}_{K+1}-\dot{x}_K\\&=\gamma\left(
\displaystyle\sum_{j\neq K+1}\zeta_{K+1}\zeta_j 
\displaystyle\frac{x_{K+1}-x_j}{2s |x_{K+1}-x_j|^{1+2s}}-\zeta_{K+1}\sigma(t,x_{K+1})\right.\\&-\left.
\displaystyle\sum_{j\neq K}\zeta_{K}\zeta_j 
\displaystyle\frac{x_{K}-x_j}{2s |x_{K}-x_j|^{1+2s}}+\zeta_{K}\sigma(t,x_{K})\right)\\&
=\gamma\left(
-\displaystyle\sum_{j=1}^{K-1}
\displaystyle\frac{1}{2s (x_{K+1}-x_j)^{2s}}-\displaystyle\frac{1}{2s (x_{K+1}-x_K)^{2s}}-\displaystyle\sum_{j=K+2}^{N}
\displaystyle\frac{1}{2s (x_j-x_{K+1})^{2s}}\right.\\&
\left.-\displaystyle\sum_{j=1}^{K-1}
\displaystyle\frac{1}{2s (x_{K}-x_j)^{2s}}-\displaystyle\frac{1}{2s (x_{K+1}-x_K)^{2s}}-\displaystyle\sum_{j=K+2}^{N}
\displaystyle\frac{1}{2s (x_j-x_{K})^{2s}}\right.\\&
\left.+\sigma(t,x_{K+1})+\sigma(t,x_{K})\right)\\&
\leq \displaystyle\gamma\left(-\frac{1}{s \theta_K^{2s}}+2\|\sigma\|_{\infty}\right).
\end{split}
\eeqs

Therefore, $\theta_K$ is subsolution of 
\beq\label{sigmagenthetasol}\dot{\theta}=-\displaystyle\frac{\gamma}{s\theta^{2s}}+2\gamma\|\sigma\|_\infty,\eeq
with initial condition 
$$\theta_K(0)=\theta_K^0.$$
If $\sigma\equiv 0$, then $\dot{\theta}_K<0$. If $\sigma\not\equiv 0$ then  
equation \eqref{sigmagenthetasol} has the stationary solution 
$\theta_s(t):=\left(\displaystyle\frac{1}{2s\|\sigma\|_\infty}\right)^\frac{1}{2s}.$
 If assumption \eqref{coo1} is satisfied, then $\theta^0_K<\theta_s$ and since
 $\theta_K$ cannot touch $\theta_s$, its derivative remains negative. Hence 
$$\theta_K\leq \theta_K^0\quad\text{and}\quad \dot{\theta}_K<
-\displaystyle\frac{\gamma}{s(\theta_K^0)^{2s}}+2\gamma\|\sigma\|_\infty<0.$$ 
As a consequence, 
there exists a finite time $T_c$ such that $\theta_K(T_c)=0.$  Since $\theta_K$ is subsolution of \eqref{sigmagenthetasol} and it is decreasing, we have  
\beqs \begin{split}s\theta_K^{2s}\dot{\theta}_K\leq-\gamma+2s\gamma\|\sigma\|_\infty\theta_K^{2s}\leq -\gamma+2s\gamma\|\sigma\|_\infty(\theta_K^0)^{2s}.
\end{split}\eeqs
Integrating in $(0,T_c)$, we get 
\beqs\frac{s}{2s+1}(\theta_K^{2s+1}(T_c)-\theta_K^{2s+1}(0))=-\frac{s}{2s+1}(\theta_K^0)^{2s+1}\leq -\gamma(1-2s\|\sigma\|_\infty(\theta_K^0)^{2s})T_c
\eeqs
which gives \eqref{T^1_cbound}.
\end{proof}
While the particles $x_K$ and $x_{K+1}$ collide at time $T_c$, the remaining particles stay at positive distance one from each other, as stated in the lemma below.
\begin{lem}
There exists $c>0$ depending on $s,N$ the $\theta_i^0$'s and $T_c$, such that, for any $t\in[0,T_c]$ and $i\neq K$, we have
\beq\label{theta++boundabove}\theta_i(t)\geq c%e^{-\gamma\|\sigma'\|_{\infty}t}
.\eeq
\end{lem}
\begin{proof}
Let us prove \eqref{theta++boundabove} for $i=1,\ldots, K-1$. Similarly one can show \eqref{theta++boundabove} for $i=K+1,\ldots,N-1$.
For $1\leq i<j\leq K$, let us denote
$$\theta_{j,i}(t):=x_j(t)-x_i(t).$$
We first show that 
\begin{equation}\label{thetaK1boundlem}\theta_{K,1}(t)\geq (x_K^0-x_1^0)e^{-\gamma\|\sigma_x\|_{\infty}t}.\eeq
Indeed, from \eqref{dynamicalsysNeven} and \eqref{kposin-knegzeta}, we have
\beqs\begin{split}
\dot{\theta}_{K,1}&=\gamma\left(
\displaystyle\sum_{l=1}^{K-1}
\displaystyle\frac{1}{2s (x_K-x_l)^{2s}} +\displaystyle\sum_{l=K+1}^{N}
\displaystyle\frac{1}{2s (x_l-x_K)^{2s}} +\displaystyle\sum_{l=2}^{K}
\displaystyle\frac{1}{2s (x_l-x_1)^{2s}}\right.\\&
\left.-\displaystyle\sum_{l=K+1}^{N}
\displaystyle\frac{1}{2s (x_l-x_1)^{2s}}-\sigma(t,x_K)+\sigma(t,x_1) \right)\\&
\geq \gamma\left(
\displaystyle\sum_{l=1}^{K-1}
\displaystyle\frac{1}{2s (x_K-x_l)^{2s}}+\displaystyle\sum_{l=2}^{K}
\displaystyle\frac{1}{2s (x_l-x_1)^{2s}}-\sigma(t,x_K)+\sigma(t,x_1) \right)\\&
\geq -\gamma \|\sigma_x\|_{\infty}\theta_{K,1},
\end{split}\eeqs
which implies \eqref{thetaK1boundlem}.

Now, suppose by contradiction that there exist $1\leq i<j\leq K$ and a first time $T>0$ such that 
\beq\label{thetajinulllem}\theta_{j,i}(T)=0.\eeq
From \eqref{thetaK1boundlem}, either $i>1$ or $j<K$. Suppose for instance $i>1$. Choose $i$ and $j$ to be respectively the minimum and the maximum index such that 
\eqref{thetajinulllem} holds, i.e., $x_{i}(T)-x_{i-1}(T)>0$ and either $x_{j+1}(T)-x_j(T)>0$ or $j=K$. Then, there exists $C_0>0$ such that for any $t\in[0,T]$, 
\beq\label{xi-xi-1wellsepalem1}-\frac{1}{2s(x_i(t)-x_{i-1}(t))^{2s}}\geq -C_0,\eeq and, if $j<K$, 
\beq\label{xi-xi-1wellsepalem2}-\frac{1}{2s(x_{j+1}(t)-x_{j}(t))^{2s}}\geq-C_0.
\eeq
Then, using \eqref{dynamicalsysNeven},  \eqref{kposin-knegzeta}, \eqref{xi-xi-1wellsepalem1} and \eqref{xi-xi-1wellsepalem2}, we get
\beqs\begin{split}
\dot{\theta}_{j,i}&=\gamma\left(\displaystyle\sum_{l=1}^{j-1}
\displaystyle\frac{1}{2s (x_j-x_l)^{2s}} -\displaystyle\sum_{l=j+1}^{K}
\displaystyle\frac{1}{2s (x_l-x_j)^{2s}}+\displaystyle\sum_{l=K+1}^{N}
+\displaystyle\frac{1}{2s (x_l-x_j)^{2s}}\right.\\&
-\displaystyle\sum_{l=1}^{i-1}
\displaystyle\frac{1}{2s (x_i-x_l)^{2s}} +\displaystyle\sum_{l=i+1}^{K}
\displaystyle\frac{1}{2s (x_l-x_i)^{2s}}-\displaystyle\sum_{l=K+1}^{N}
+\displaystyle\frac{1}{2s (x_l-x_i)^{2s}}\\&
\left.-\sigma(t,x_j)+\sigma(t,x_i)\right)\\&
\geq \gamma\left(\displaystyle\sum_{l=1}^{j-1}
\displaystyle\frac{1}{2s (x_j-x_l)^{2s}} -\displaystyle\sum_{l=j+1}^{K}
\displaystyle\frac{1}{2s (x_l-x_j)^{2s}}\right.\\&
-\displaystyle\sum_{l=1}^{i-1}
\displaystyle\frac{1}{2s (x_i-x_l)^{2s}} +\displaystyle\sum_{l=i+1}^{K}
\displaystyle\frac{1}{2s (x_l-x_i)^{2s}}\\&
\left.-\sigma(t,x_j)+\sigma(t,x_i)\right)\\&
\geq \gamma\left(\displaystyle\frac{1}{s \theta_{j,i}^{2s}} -\displaystyle\sum_{l=j+1}^{K}
\displaystyle\frac{1}{2s (x_l-x_j)^{2s}}-\displaystyle\sum_{l=1}^{i-1}
\displaystyle\frac{1}{2s (x_i-x_l)^{2s}}-\sigma(t,x_j)+\sigma(t,x_i)\right)\\&
\geq \gamma\left(\displaystyle\frac{1}{s \theta_{j,i}^{2s}}-C- \|\sigma_x\|_{\infty}\theta_{j,i}\right),
\end{split}\eeqs
where $C=(K-j+i-1)C_0$. 
%Therefore, the function
%$$\widetilde{\theta}_{j,i}(t):=\theta_{j,i}(t)e^{\gamma\|\sigma'\|_{\infty}t},$$ satisfies
%$$\dot{\widetilde{\theta}}_{j,i}\geq \gamma\left(\displaystyle\frac{1}{s \theta_{j,i}^{2s}}-C\right)$$
Now, \eqref{thetajinulllem} implies that for any $\delta>0$ there exists $t_\delta>0$ such that $0<\theta_{j,i}(t)\leq\delta$ for any $t\in(T-t_\delta,T)$. Choosing $\delta$ small enough so that
$$ \displaystyle\frac{1}{s\delta^{2s}}-C- \|\sigma_x\|_{\infty}\delta>0,$$ from the computation above we see that $\theta_{j,i}$ is increasing in $ (T-t_\delta,T)$ and this contradicts  \eqref{thetajinulllem}.
Estimate \eqref{theta++boundabove} for $i<K$ is then proven. A similar argument gives  \eqref{theta++boundabove} when $i>K$. \end{proof}

Now, as  firstly seen in \cite{gonzalezmonneau,dpv, dfv,pv3}, we consider an auxiliary small parameter $\delta>0$ and define $(\xs_1(t),\ldots,\xs_N(t))$ to be the solution to the following system: for $i=1,\ldots,N$ 

\beq\label{dynamicalsysbarNeven}\begin{cases} \dot{\xs}_i=\gamma\left(
\displaystyle\sum_{j\neq i}\zeta_i\zeta_j 
\displaystyle\frac{\xs_i-\xs_j}{2s |\xs_i-\xs_j|^{1+2s}}-\zeta_i\sigma(t,\xs_i)-\zeta_i\delta\right)&\text{in }(0,T^\delta_c)\\
 \xs_i(0)=x_i^0-\zeta_i\delta,
\end{cases}\eeq
where the $\zeta_i$'s are given by \eqref{kposin-knegzeta} and $T^\delta_c$ is the collision time of the  perturbed system \eqref{dynamicalsysbarNeven}.
Let us denote for $i=1,\ldots,N-1$
\beq\label{thetabarNeven}\overline{\theta}_i(t):=\xs_{i+1}(t)-\xs_i(t).\eeq 
%then $(\overline{\theta}_1,\ldots, \overline{\theta}_{N-1})$ is solution of
%\beq\label{thetaeqNeven}\begin{cases} \dot{\theta}_i=\displaystyle\frac{\gamma}{2s}\left(\frac{2\zeta_i\zeta_{i+1}}{\theta_i^{2s}}+\displaystyle\sum_{j=1}^{i-1}\zeta_{i+1}\zeta_j %\displaystyle\frac{1}{ (x_{i+1}-x_j)^{2s}}-\displaystyle\sum_{j=i+2}^{N}\zeta_{i+1}\zeta_j \displaystyle\frac{1}{ (x_j-x_{i+1})^{2s}}\right.\\
% \quad\left. -\displaystyle\sum_{j=1}^{i-1}\zeta_{i}\zeta_j \displaystyle\frac{1}{ (x_{i}-x_j)^{2s}}+\displaystyle\sum_{j=i+2}^{N}\zeta_{i}\zeta_j \displaystyle\frac{1}{ (x_j-x_{i})^{2s}} \right)&\text{in }(0,T_c)\\
 %\theta_i(0)=\theta_i^0,
%\end{cases}\eeq
The following results have been proven in \cite{pv3} in the case $N=3$. Since the proofs do not change in the case $N>3$, we skip them and we refer to the analogous results in \cite{pv3}.
\begin{prop}\label{tclimpropNeven}
Let $(x_1,\ldots,x_N)$ and $(\xs_1,\ldots,\xs_N)$ be the solution respectively of system~\eqref{dynamicalsysNeven} and~\eqref{dynamicalsysbarNeven}. 
Let   $T_c<+\infty$ and $T_c^\delta$  be the collision time  respectively of~\eqref{dynamicalsysNeven} and~\eqref{dynamicalsysbarNeven}. Then we have 
\beq\label{TcdeltalimNeven}\lim_{\delta\to 0}T_c^\delta=T_c,\eeq and for $i=1,\ldots,N$
\beq\label{xideltalim}\lim _{\delta\to 0}\xs_i(t)=x_i(t)\quad\text{for any }t\in[0,T_c).\eeq
\end{prop}
\begin{proof}
See the proof of Proposition 5.1 in \cite{pv3}.
\end{proof}

\begin{prop}\label{holderthetapropNeven} Let  $(\xs_1,\ldots,\xs_N)$ be the solution to  system  \eqref{dynamicalsysbarNeven} and $(\overline{\theta}_1,\ldots,\overline{\theta}_{N-1})$ given by
\eqref{thetabarNeven}. Then, for any $0\leq\delta\leq 1$ the function $
\min_{i=1,\ldots,N}\overline{\theta}_i$  is H\"{o}lder continuous in $[0,T_c^\delta]$ with H\"{o}lder constant uniform in $\delta$.
\end{prop}
\begin{proof}
See the proof of Proposition 5.2 in \cite{pv3}.
\end{proof}

\noindent Next,  we set 
\beq\label{xbarpuntoNeven}\cs_i(t):= \dot{\xs}_i(t),\quad i=1,\ldots,N\eeq
and
\beq\label{barsigmaNevenbars}
\overline{\sigma}:=\displaystyle\frac{\sigma+\delta}{\beta},\eeq
where $\beta$ is given by \eqref{beta}.
Let $u$ and $\psi$ be  respectively the solution of \eqref{u} and \eqref{psi}. We define
\beq\label{vepansbarNeven}\begin{split}\vs_\ep(t,x)&:=\ep^{2s}\overline{\sigma}(t,x)+\sum_{i=1}^N u\left(\displaystyle\zeta_i\frac{x-\xs_i(t)}{\ep}\right)-(N-K)
-\sum_{i=1}^N\zeta_i\ep^{2s}\cs_i(t)\psi\left(\displaystyle\zeta_i\frac{x-\xs_i(t)}{\ep}\right).\end{split}
\eeq
The situation is depicted in Figure~5.
\bigskip

\begin{center}
  \includegraphics[width=1.01\textwidth]{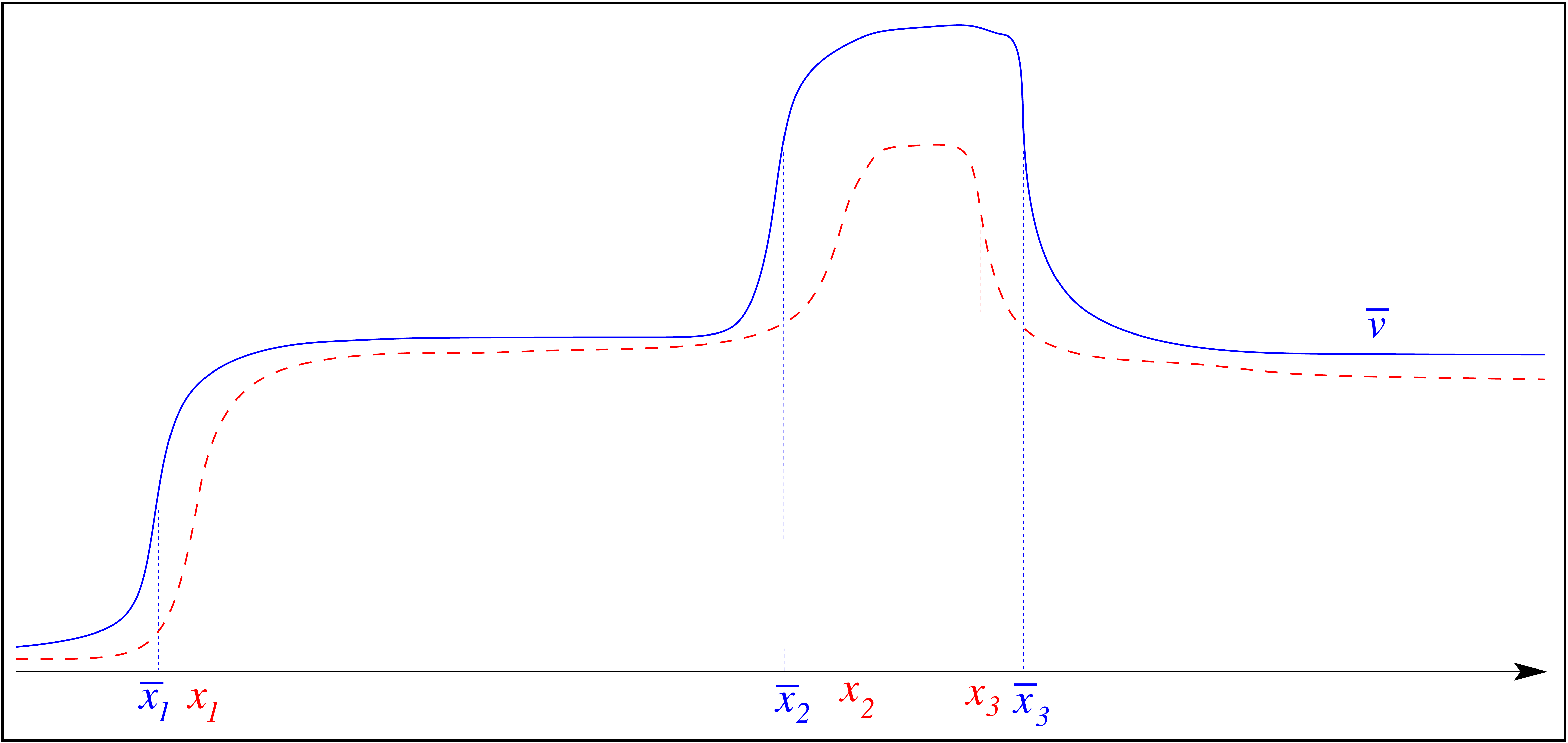}
{\footnotesize{
\nopagebreak$\,$\\
\nopagebreak
Figure 5: The barrier
of Proposition~\ref{thetaepropNeven}.}}
\end{center}

\bigskip
\bigskip

Under the appropriate choice of the parameters,
the function $\vs_\ep$ is a supersolution of \eqref{vepeq}-\eqref{vep03}, as next results point out:

\begin{prop}\label{thetaepropNeven}  There exist $\ep_0>0$ and $\theta_\ep,\,\delta_\ep>0$ with 
\beq\label{thetaepthetaprop} \theta_\ep,\,\delta_\ep,\,\ep\theta_\ep^{-2}=o(1)\quad\text{as }\ep\to0\eeq
such that for any $\ep<\ep_0$, if $(\xs_1,\ldots,\xs_N)$ is a solution of the 
ODE system in \eqref{dynamicalsysbarNeven} with $\delta\ge\delta_\ep$, then the  function $\vs_\ep$ defined in \eqref{vepansbarNeven} satisfies
$$\ep(\vs_\ep)_t-\I \vs_\ep+\displaystyle\frac{1}{\ep^{2s}}W'(\vs_\ep)-\sigma\geq 0$$
%$$\lim_{\ep\to0}\theta_\ep=\lim_{\ep\to0}\delta_\ep=0$$ and 
for any $x\in\R$ and  any $t\in(0,T_c^\delta)$ such that $\xs_{i+1}(t)-\xs_i(t)\geq \theta_\ep$ for $i=1,\ldots,N-1$.
\end{prop}
\begin{proof}
See the proof of Proposition 5.3 in \cite{pv3}.
\end{proof}

\begin{lem}\label{initialcondpropNeven}  Let $v_\ep^0(x)$ be defined by \eqref{vep03}. Then there exists $\ep_0>0$ such that for any $\ep<\ep_0$ and $\delta_\ep$ given by Proposition \ref{thetaepropNeven}, if $(\xs_1,\ldots,\xs_N)$ is the solution to system  \eqref{dynamicalsysbarNeven} with $\delta=\delta_\ep$, then  
the function $\vs_\ep$ defined in \eqref{vepansbarNeven} satisfies
$$v_\ep^0(x)\le \vs_\ep(0,x)\quad\text{for any }x\in\R.$$
\end{lem}
\begin{proof}
See the proof of Lemma 5.4 in \cite{pv3}.
\end{proof}

\noindent Now we consider the barrier function $\vs_\ep$ defined in \eqref{vepansbarNeven}, where  $(\xs_1,\ldots,\xs_N)$ is the solution to  system  \eqref{dynamicalsysbarNeven} in which we fix
 $\delta=\delta_\ep$, with $\delta_\ep$ given by Proposition \ref{thetaepropNeven}. 
 For $\ep$ small enough,  from \eqref{TcdeltalimNeven}, \eqref{xideltalim} and \eqref{theta++boundabove}, we infer that 
  there exists 
  $T^1_\ep>0$ such that 
  \beq\label{xsepk+1-xsepk=thetaepNeven}\overline{\theta}_K(T^1_\ep)= \xs_{K+1}(T^1_\ep)-\xs_K(T^1_\ep)=\theta_\ep,\eeq and
  $$\overline{\theta}_K(t)=\xs_{K+1}(t)-\xs_K(t)>\theta_\ep\quad\text{for any }t<T^1_\ep,$$ and there exists a constant $c_0>0$ independent of $\epsilon$ such that
  \beq\label{thetabarbounbelow} \xs_{i+1}(t)-\xs_i(t)\geq c_0\quad\text{for any }t\leq T^1_\ep\text{ and }i\neq K.\eeq
From \eqref{dynamicalsysbarNeven}, \eqref{xbarpuntoNeven} and \eqref{xsepk+1-xsepk=thetaepNeven}, we infer that 
 \beq\label{concmainhm33barbar}  |\cs_K(T^1_\ep)|\leq C\theta_\ep^{-2s}.\eeq
\noindent  By Proposition \ref{thetaepropNeven} and Lemma \ref{initialcondpropNeven},  the function~$\vs_\ep$ defined in \eqref{vepansbarNeven}, is a
supersolution of \eqref{vepeq}-\eqref{vep03} in $(0,T^1_\ep)\times\R$, and the comparison principle implies 
 \beq\label{veplessvepbarNeven} v_\ep(t,x)\le \vs_\ep(t,x)\quad\text{for any }(t,x)\in [0,T^1_\ep]\times\R.\eeq

\noindent  Moreover, since $\theta_\ep=o(1)$ as $\ep\to0$,  we have 
 \beq\label{T1elim}T^1_\ep=T_c^1+o(1)\quad\text{as }\ep\to 0.\eeq 
 Indeed, if up to subsequences, $T^1_\ep$ converges as $\ep\to 0$ to some $T>0$, since $T^1_\ep\leq T_c^{\delta_\ep}$ then by \eqref{TcdeltalimNeven} we have $T\leq T_c^1$. Suppose by contradiction that 
 \beq\label{tlesstccontr}T< T_c^1.\eeq Then by Proposition \ref{holderthetapropNeven} and \eqref{xsepk+1-xsepk=thetaepNeven}
 $$|\overline{\theta}_K(T^1_\ep)-\overline{\theta}_K(T)|=|\theta_\ep-\overline{\theta}_K(T)|\leq C|T^1_\ep-T|^\alpha,$$ for some $C>0$ and $\alpha\in(0,1)$ independent of $\epsilon$. This and 
 \eqref{xideltalim} imply that $\theta_K(T)=0$ which is in contradiction with \eqref{tlesstccontr}. Thus \eqref{T1elim} is proven.

 Next, to conclude the proof of Theorem \ref{mainthmbeforecollNeven}, we are going to show that starting from $T^1_\ep$, after a small time $t_\ep$, the function  $v_\ep$ satisfies
 \eqref{vlesepN-2}, 
%\beq\label{veplessvarrhoproofNeven}v_\ep(t,x)\le\varrho_\ep+\eeq 
for some $\varrho_\ep=o(1)$ and some  $x_i^\ep=x_i(T_c)+o(1)$, $i\neq K,\,K+1$,  as $\ep\to0$. 
For this scope, we denote
\beqs \xs_i^\ep:=\xs_i(T^1_\ep),\quad i=1,\ldots, N.\eeqs We recall that from \eqref{xsepk+1-xsepk=thetaepNeven}
\beq\label{xsepk+1-xsepk=thetaepNevensecond} \xs_{K+1}^\ep-\xs_K^\ep=\theta_\ep.\eeq
We show  \eqref{vlesepN-2} for $x\leq  \xs_K^\ep+\frac{\theta_\ep}{2}$, similarly one can prove  it for $x\geq  \xs_K^\ep+\frac{\theta_\ep}{2}$.
For this aim let us introduce   the following  further perturbed system: for $\dss>\delta_\ep$,  $L>1$ such that $\xs_{K+1}^\ep+L\theta_\ep<\xs_{K+2}^\ep$, and  $\zeta_i$'s  given by \eqref{kposin-knegzeta},

 \beq\label{dynamicalsysbarbarNeven}\begin{cases} \dot{\xss}_i=\gamma\left(
\displaystyle\sum_{j\neq i}\zeta_i\zeta_j 
\displaystyle\frac{\xss_i-\xss_j}{2s |\xss_i-\xss_j|^{1+2s}}-\zeta_i\sigma(t,\xss_i)-\zeta_i\dss\right)&\text{in }(0,T^{\dss}_c)\\
\xss_{K+1}(0)=\xs_{K+1}^\ep+L\theta_\ep\\
 \xss_i(0)=\xs_i^\ep-\zeta_i\theta_\ep&i\neq K+1,
\end{cases}\eeq
where $T^{\dss}_c$ is the collision time of the system \eqref{dynamicalsysbarbarNeven}.
We set 
\beq\label{xbarbarpuntoNeven}\css_i(t):= \dot{\xss}_i(t),\quad i=1,\ldots,N\eeq
and
\beq\label{barsigmaNeven}
\hat{\sigma}:=\displaystyle\frac{\sigma+\dss}{\beta},\eeq
where $\beta$ is given by \eqref{beta}.

\noindent We define
\beq\label{vepansbarbarNeven}\begin{split}\hat{v}_\ep(t,x)&:=\ep^{2s}\hat{\sigma}(t,x)+\sum_{i=1}^N u\left(\displaystyle\zeta_i\frac{x-\xss_i(t)}{\ep}\right)-(N-K)
-\sum_{i=1}^N\zeta_i\ep^{2s}\css_i(t)\psi\left(\displaystyle\zeta_i\frac{x-\xss_i(t)}{\ep}\right),\end{split}
\eeq
where again $u$ and $\psi$ are  respectively the solution of \eqref{u} and \eqref{psi}. 
\begin{lem}\label{vbarelessvtilteinitialtimelem} There exist $\ep_0,\,\dss_\ep>0$ with  $ \delta_\ep<\dss_\ep= \delta_\ep+o(1)$ as $\ep\to0$, where $\delta_\ep$ is given by Proposition \ref{thetaepropNeven}, such that  if  $(\xss_1,\ldots,\xss_N)$ is the solution to  system  \eqref{dynamicalsysbarbarNeven} with $\dss=\dss_\ep$, then 
the function $\hat{v}_\ep$ defined in \eqref{vepansbarbarNeven} satisfies
\beqs\hat{v}_\ep(0,x)\ge \vs_\ep(T_\ep^1,x)\quad\text{for any }x\in\R.\eeqs
\end{lem}
\begin{proof}
See the proof of Lemma 5.6 in \cite{pv3}.
\end{proof}

\begin{lem}\label{teplem}Let
\beq\label{tep} t_\ep:=\frac{4s(L+2)^{2s}\theta_\ep^{2s+1}}{\gamma[1-2s(L+2)^{2s}\theta_\ep^{2s}(\|\sigma\|_\infty+\dss)]}.\eeq
Then there exists $L>1$, $c_1>0$, and $\ep_0, \dss_0>0$ such that for any $\ep<\ep_0$  and $\dss<\dss_0$, 
\beq\label{xsk+l<xsk+1}\xs_{K+1}^\ep+L\theta_\ep<\xs_{K+2}^\ep,\eeq
and 
the solution $(\xss_1,\ldots,\xss_N)$  to  system  \eqref{dynamicalsysbarbarNeven} satisfies
\beq\label{xss1graterxs2lem} \xss_K(t_\ep)\geq \xs_{K+1}^\ep, \eeq for any $t\in[0,t_\ep]$, $\xss_{K+1}(t)-\xss_K(t)$ is decreasing and 
\beq\label{xss2-xss1tep} \xss_{K+1}(t)-\xss_K(t)\geq \xss_{K+1}(t_\ep)-\xss_K(t_\ep)\geq \theta_\ep,\eeq
and for any $t\in[0,t_\ep]$ and $i\neq K$
\beq\label{xssI=1-xssitep} \xss_{i+1}(t)-\xss_i(t) \geq c_1.\eeq
%\beq\label{xss3-xss2tep}\xss_3(t_\ep)-\xss_2(t_\ep).\eeq
\end{lem}
\begin{proof} Let us denote
\beqs \hat{\theta}_K(t):=\xss_{K+1}(t)-\xss_K(t).\eeqs 
Then, from  Lemma \ref{T^1_cboundlemma}, for $\ep$ and $\dss$ small enough, 
 such that 
\beqs\thh_K(0)=(L+2)\theta_\ep<\left[\frac{1}{2s(\|\sigma\|_\infty+\dss)}\right]^\frac{1}{2s},\eeqs
$\thh_K$ is decreasing, therefore for $t>0$,
\beq\label{theta1decreqsecond2}\thh_K(t)<(L+2)\theta_\ep.\eeq
Moreover, 
 there exists $\tau>0$ satisfying
\beq\label{uppertau}\tau<\frac{s(L+2)^{1+2s}\theta_\ep^{1+2s}}{(2s+1)\gamma(1-2s(\theta_\ep)^{2s}(L+2)^{2s}(\|\sigma\|_\infty+\dss))},\eeq
such that $\hat{\theta}_K(\tau)=\theta_\ep$. 
Remark that $\tau=o(1)$ as $\ep\to0$, then 
from     \eqref{thetabarbounbelow} we infer that, for $\epsilon$  and $\dss$ small enough, there exists a constant $c_1$ independent of $\epsilon$ and $\dss$ such that 
\beq\label{thhineqkteplem} \thh_i(t)\geq c_1\quad\text{for any }t\in[0,\tau]\text{ and }i\neq K.\eeq 
From \eqref{thhineqkteplem} (where the $\zeta_i$'s are given by \eqref{kposin-knegzeta}) and  \eqref{dynamicalsysbarbarNeven}, we infer that 
\beqs\begin{split}
\dot{\thh}_K&=\gamma\left(-\displaystyle\sum_{j=1}^{K-1}
\displaystyle\frac{1}{2s (\xss_{K+1}-\xss_j)^{2s}}-\displaystyle\frac{1}{2s (\xss_{K+1}-\xss_K)^{2s}}-\displaystyle\sum_{j=K+2}^{N}
\displaystyle\frac{1}{2s (\xss_j-\xss_{K+1})^{2s}}\right.\\&
\left.-\displaystyle\sum_{j=1}^{K-1}
\displaystyle\frac{1}{2s (\xss_{K}-\xss_j)^{2s}}-\displaystyle\frac{1}{2s (\xss_{K+1}-\xss_K)^{2s}}-\displaystyle\sum_{j=K+2}^{N}
\displaystyle\frac{1}{2s (\xss_j-\xss_{K})^{2s}}\right.\\&
\left.+\sigma(t,\xss_{K+1})+\sigma(t,\xss_{K})+2\dss\right)\\&
\geq \gamma\left(-\frac{1}{s \thh_K^{2s}}-\displaystyle\sum_{j=1}^{K-1}
\displaystyle\frac{1}{2s (\thh_K+\ldots+\thh_j)^{2s}}-\displaystyle\sum_{j=K+2}^{N}
\displaystyle\frac{1}{2s (\thh_{K+1}+\ldots+\thh_{j-1})^{2s}}\right.\\&
\left.-\displaystyle\sum_{j=1}^{K-1}
\displaystyle\frac{1}{2s (\thh_{K-1}+\ldots+\thh_{j})^{2s}}-\displaystyle\sum_{j=K+2}^{N}
\displaystyle\frac{1}{2s (\thh_K+\ldots+\thh_{j-1})^{2s}}-2\|\sigma\|_{L^\infty}\right)\\&
\geq \gamma\left(-\frac{1}{s \thh_K^{2s}}-C-2\|\sigma\|_{L^\infty}\right),
\end{split}
\eeqs
for some $C>0$ independent of $\ep$ and $\dss$.

Combining the previous estimate with \eqref{theta1decreqsecond2}, we get, for any $t\in(0,\tau)$,
$$\dot{\hat{\theta}}_K
\geq\gamma\left(\displaystyle\frac{-1-(2\|\sigma\|_\infty +C) s\hat{\theta}_K^{2s}}{s \hat{\theta}_K^{2s}}\right)
\geq\gamma\left(\displaystyle\frac{-1-(2\|\sigma\|_\infty+C) s(L+2)^{2s}\theta_\ep^{2s}}{s \hat{\theta}_K^{2s}}\right),$$
i.e.,
$$\hat{\theta}_K^{2s}\dot{\hat{\theta}}_K\geq \frac{\gamma}{s}\left(-1-(2s\|\sigma\|_\infty +C)(L+2)^{2s}\theta_\ep^{2s}\right).$$
Integrating the previous inequality in $(0,\tau)$, we get
\beqs\begin{split}\frac{1}{2s+1}(\thh^{2s+1}(\tau)-\thh^{2s+1}(0))&=\frac{1}{2s+1}\theta_\ep^{2s+1}(1-(L+2)^{2s+1})\\&
\geq \frac{\gamma}{s}\left(-1-(2s\|\sigma\|_\infty+C) (L+2)^{2s}\theta_\ep^{2s}\right)\tau,
\end{split}
\eeqs from which 
\beq\label{tauupper}\tau\geq \frac{s\theta_\ep^{2s+1}[(L+2)^{2s+1}-1]}{\gamma(2s+1)(1+(2s\|\sigma\|_\infty +C)(L+2)^{2s}\theta_\ep^{2s})}.\eeq

Next,  \eqref{dynamicalsysbarbarNeven} and \eqref{theta1decreqsecond2} imply 
\beq\label{xss1inelemep2}\begin{split}  \dot{\xss}_K=&\gamma\left(\displaystyle\sum_{j=1}^{K-1} 
\displaystyle\frac{1}{2s (\xss_{K}-\xss_j)^{2s}} +\frac{1}{2s (\xss_{K+1}-\xss_K)^{2s}}+\displaystyle\sum_{j=K+2}^{N} 
\displaystyle\frac{1}{2s (\xss_{j}-\xss_K)^{2s}} -\sigma(t,\xss_K)-\dss\right)\\&
\geq \gamma\left(\displaystyle\frac{1}{2s \thh_K^{2s}} -\sigma(t,\xss_K)-\dss\right)\\&
\geq \gamma\left(\displaystyle\frac{1}{2s(L+2)^{2s} (\theta_\ep)^{2s}}-\|\sigma\|_\infty-\dss\right)\\&>0.
\end{split}
\eeq
Let  $t$ be the time such that $\xss_K(t)= \xs_{K+1}^\ep=\xss_K(0)+2\theta_\ep$, then integrating \eqref{xss1inelemep2} in $(0,t)$ we get 
\beqs \xss_K(t)-\xss_K(0)=2\theta_\ep\geq \gamma\left(\frac{1}{2s(L+2)^{2s}\theta_\ep^{2s}}-\|\sigma\|_\infty-\dss\right)t,\eeqs
from which 
\beq\label{tlessteplemtep2}
t\leq t_\ep\eeq where $t_\ep$ is defined by \eqref{tep}.

Comparing $\tau$ with $t_\ep$, from   \eqref{tep} and \eqref{tauupper}, we see that it is possible to choose $L$ big enough so that $$\tau>t_\ep\geq t.$$
Estimate \eqref{thhineqkteplem}  and $\tau>t_\ep$ imply \eqref{xssI=1-xssitep}. 
Moreover,  from  \eqref{thetabarbounbelow},  for any fixed $L$, we can choose $\ep$ small enough so that \eqref{xsk+l<xsk+1} holds.
For such a choice of $L$, the decreasing monotonicity of $\hat{\theta}_K$   implies \eqref{xss2-xss1tep}. Finally,  
 \eqref{tlessteplemtep2} and the increasing monotonicity of $\xss_K$ give
 $$\xss_K(t_\ep)\geq \xss_K(t)= \xs_{K+1}^\ep,$$ which proves \eqref{xss1graterxs2lem}. This concludes the proof of the lemma.\end{proof}

We consider now as barrier the function  $\hat{v}_\ep$ defined in \eqref{vepansbarbarNeven}, where we fix $\dss=\dss_\ep$ in system  \eqref{dynamicalsysbarbarNeven}, with 
$\dss_\ep$ given by Lemma \ref{vbarelessvtilteinitialtimelem}, and $L$ given by Lemma \ref{teplem}. 
For $\ep$ small enough, from \eqref{xss2-xss1tep}, \eqref{xssI=1-xssitep}  and  Proposition \ref{thetaepropNeven}, the function $\hat{v}_\ep$  satisfies 
$$\ep(\hat{v}_\ep)_t-\I \hat{v}_\ep+\displaystyle\frac{1}{\ep^{2s}}W'(\hat{v}_\ep)-\sigma(t,x)\geq0\quad\text{in }(0,t_\ep)\times\R$$ where $t_\ep$ is given by \eqref{tep}.
Moreover from \eqref{veplessvepbarNeven} and Lemma \ref{vbarelessvtilteinitialtimelem} 
\beqs v_\ep(T_\ep^1,x)\le\hat{v}_\ep(0,x)\quad\text{for any }x\in\R.\eeqs
The comparison principle then implies 
\beq\label{concmainhm1} v_\ep(T_\ep^1+t,x)\le\hat{v}_\ep(t,x)\quad\text{for any }(t,x)\in[0,t_\ep]\times\R.\eeq
Now, for $x\leq  \xs_K^\ep+\frac{\theta_\ep}{2}$, from \eqref{xsepk+1-xsepk=thetaepNeven}, \eqref{xss1graterxs2lem} and \eqref{xss2-xss1tep}   we know that 
$$x-\xss_K(t_\ep)\leq-\frac{\theta_\ep}{2}\quad\text{and}\quad \xss_{K+1}(t_\ep)-x\ge \frac{3\theta_\ep}{2}.$$ Therefore, from estimate \eqref{uinfinity} we have

\beq\label{concmainhm2Neven} u\left(\displaystyle\frac{x-\xss_K(t_\ep)}{\ep}\right)+ u\left(\displaystyle\frac{\xss_{K+1}(t_\ep)-x}{\ep}\right)-1\leq C\ep^{2s}\theta_\ep^{-2s}.\eeq
Moreover \eqref{xss2-xss1tep}, \eqref{xssI=1-xssitep}, \eqref{dynamicalsysbarbarNeven} and \eqref{xbarbarpuntoNeven} imply that for $i=K,K+1$
\beq\label{concmainhm3Neven}  |\css_i(t_\ep)|\leq C\theta_\ep^{-2s}.\eeq
From the \eqref{concmainhm2Neven}, \eqref{concmainhm3Neven} and \eqref{concmainhm1}, we infer that, for $x\leq  \xs_K^\ep+\frac{\theta_\ep}{2}$, we have 
\beq\label{veplessN-2barrier}\begin{split} v_\ep(T_\ep^1+t_\ep,x)&\leq \ep^{2s}\hat{\sigma}(t_\ep,x)+\sum_{i=1\atop i\neq K,K+1}^N u\left(\displaystyle\zeta_i\frac{x-\xss_i(t_\ep)}{\ep}\right)-(N-K-1)\\&
-\sum_{i=1\atop i\neq K,K+1}^N\zeta_i\ep^{2s}\css_i(t)\psi\left(\displaystyle\zeta_i\frac{x-\xss_i(t_\ep)}{\ep}\right)+C\ep^{2s}\theta_\ep^{-2s}\\&
\leq \frac{\ep^{2s}}{\beta}\sigma(t_\ep,x)+\sum_{i=1\atop i\neq K,K+1}^N u\left(\displaystyle\zeta_i\frac{x-\xss_i(t_\ep)}{\ep}\right)-(N-K-1)
+\varrho_\ep,\end{split}\eeq
where $$\varrho_\ep=O(\ep^{2s}\theta_\ep^{-2s}).$$ 
From \eqref{thetaepthetaprop}, we see that $\varrho_\ep$ satisfies \eqref{varroepmainthm}.

  Similarly, one can prove that 
\beq\label{concmainhm1bis} v_\ep(T_\ep^1+t,x)\le{\hat{w}}_\ep(t,x)\quad\text{for any }(t,x)\in[0,t_\ep]\times\R.\eeq
where $\hat{w}_\ep$ is defined by 
\beqs%\label{vepansbarbarNevenbis}
\begin{split}\hat{w}_\ep(t,x)&:=\ep^{2s}\hat{\sigma}(t,x)+\sum_{i=1}^N u\left(\displaystyle\zeta_i\frac{x-\hat{y}_i(t)}{\ep}\right)-(N-K)
-\sum_{i=1}^N\zeta_i\ep^{2s}\hat{d}_i(t)\psi\left(\displaystyle\zeta_i\frac{x-\hat{y}_i(t)}{\ep}\right),\end{split}
\eeqs
where $(\hat{y}_1,\ldots,\hat{y}_N)$ is the solution of the system \eqref{dynamicalsysbarbarNeven} with initial condition
\begin{equation}\label{initailyi}\begin{split}&\hat{y}_i(0)=\xs_i^\ep-\zeta_i\theta_\ep, \quad i\neq K,\\&
\hat{y}_K(0)=\xs_K^\ep-L\theta_\ep,
\end{split}\end{equation} for $L$ large enough, small $\ep$ and $\dss=\dss_\ep$, and
\beqs \hat{d}_i:=\dot{\hat{y}}_i(t),\quad i=1,\ldots,N.\eeqs

\noindent  As before, from \eqref{concmainhm1bis}, we get that, for $x\geq  \xs_K^\ep+\frac{\theta_\ep}{2}$,
 \beq\label{veplessN-2barrierbis}\begin{split} v_\ep(T_\ep^1+t_\ep,x)&\leq \frac{\ep^{2s}}{\beta}\sigma(t_\ep,x)+\sum_{i=1\atop i\neq K,K+1}^N u\left(\displaystyle\zeta_i\frac{x-\hat{y}_i(t_\ep)}{\ep}\right)-(N-K-1)
+\varrho_\ep.\end{split}\eeq
Now, from \eqref{thetabarbounbelow}, \eqref{dynamicalsysbarbarNeven} and \eqref{initailyi}, we see that 
\beq\label{xi-yioep}|x_i(t_\ep)-y_i(t_\ep)|=o(1)\quad\text{as }\ep \to 0,\text{ for }i\neq K,K+1.\eeq
Estimates \eqref{xi-yioep} combined with  \eqref{thetabarbounbelow}, 
imply that there exists a constant $c>0$ independent of $\ep$ such that, for $i\neq K,K+1$,
\beqs\begin{split}&\max(x_{i-1}(t_\ep),y_{i-1}(t_\ep))+c\leq \min(x_i(t_\ep),y_i(t_\ep))\\& \leq \max(x_i(t_\ep),y_i(t_\ep))\leq \min(x_{i+1}(t_\ep),y_{i+1}(t_\ep))-c.\end{split}
\eeqs
Therefore, if we  define
$$x_i^\ep:=\begin{cases}\min(x_i(t_\ep),y_i(t_\ep))&\text{for }i=1,\ldots,K-1\\
\max(x_i(t_\ep),y_i(t_\ep))&\text{for }i=K+1,\ldots,N,
\end{cases}$$ 
%$$x_i^\ep:=\max(x_{i+2}(t_\ep),y_{i+2}(t_\ep))\quad\text{for }i=K,\ldots,N-2,$$ 
we see that the $x_i^\ep$'s satisfy  \eqref{xi+1-xiboundmainthm1}. Moreover, for $x\leq  \xs_K^\ep+\frac{\theta_\ep}{2}$,
 from \eqref{veplessN-2barrier}, \eqref{kposin-knegzeta} and the monotonicity of $u$ we get 

\beqs\begin{split} v_\ep(T_\ep^1+t_\ep,x)&\leq 
\frac{\ep^{2s}}{\beta}\sigma(t_\ep,x)+\sum_{i=1\atop i\neq K,K+1}^N u\left(\displaystyle\zeta_i\frac{x-\xss_i(t_\ep)}{\ep}\right)-(N-K-1)
+\varrho_\ep\\&
=\frac{\ep^{2s}}{\beta}\sigma(t_\ep,x)+\sum_{i=1}^{K-1}u\left(\displaystyle\frac{x-\xss_i(t_\ep)}{\ep}\right)+\sum_{i=K+2}^{N}u\left(\displaystyle\frac{\xss_i(t_\ep)-x}{\ep}\right)\\&-(N-K-1)
+\varrho_\ep\\&
\leq \ep^{2s}\hat{\sigma}(t_\ep,x)+\sum_{i=1}^{K-1}u\left(\displaystyle\frac{x-x_i^\ep}{\ep}\right)+\sum_{i=K+2}^{N}u\left(\displaystyle\frac{x_i^\ep-x}{\ep}\right)\\&-(N-K-1)
+\varrho_\ep,
\end{split}\eeqs
which gives  \eqref{vlesepN-2} for $x\leq  \xs_K^\ep+\frac{\theta_\ep}{2}$. Similarly, from \eqref{veplessN-2barrierbis} and the monotonicity of $u$ we get  \eqref{vlesepN-2} for $x\geq  \xs_K^\ep+\frac{\theta_\ep}{2}$. Estimates \eqref{varroepmainthm}  follow from \eqref{thetaepthetaprop}. This concludes the proof of Theorem \ref{mainthmbeforecollNeven}.

%%%%%%%%%%%%%%%%%%%%%%%%%%%%%%%%
%%%%%%%%%%%%%%%%%%%%%%%%%%%%%%%%%%

\section{Proof of Theorem \ref{thmexponentialdecayNeven}}\label{R F 24}

%This section is devoted to the proof of Theorem \ref{thmexponentialdecay3}. 
Let us consider the function
 \beq\label{h}h(t,x):=\frac{\ep^{2s}}{\beta}\sigma(t,x) +\sum_{i=1\atop i\neq K,K+1}^{N}u\left(\zeta_i\frac{x-x_i(t)}{\ep}\right)-(N-K-1)+\varrho_\ep e^{-\frac{\mu t}{\ep^{2s+1}}}\eeq
 where
 \beq\label{x(t)}x_i(t):=x_i^\ep+\zeta_iK_\ep \varrho_\ep (e^{-\frac{\mu t}{\ep^{2s+1}}}-1),\eeq
 the $x_i^\ep$'s  and $\varrho_\ep$ are given by Theorem \ref{mainthmbeforecollNeven} and the $\zeta_i$'s satisfy \eqref{kposin-knegzeta}
We show that, choosing conveniently $K_\ep$ and $\mu$, ~$h$ is a supersolution of the equation \eqref{vepeq} for small times, as next result states: 
 
\begin{lem}\label{layerexponelem}There exist $\ep_0>0$ and $\mu>0$,  such that for any $\ep< \ep_0$, there exist $K_\ep,\,\tau_\ep>0$  such that  
\beq\label{varrotildelem} \varrho_\ep K_\ep=o(1),\quad\tau_\ep=o(1),\quad \varrho_\ep e^{-\frac{\mu \tau_\ep}{\ep^{2s+1}}}=\ep^{2s}o(1)\quad\text{as }\ep\to 0,\eeq 
and the
function $h$ defined in \eqref{h}-\eqref{x(t)} satisfies
$$\ep h_t-\I h+\displaystyle\frac{1}{\ep^{2s}}W'(h)-\sigma(t,x)\geq 0\quad\text{for any }(t,x)\in (0,\tau_\ep)\times\R.$$
\end{lem}
\begin{proof}
Let $\alpha,\,\gamma\in(0,1)$ be such that 
\beq\label{alphaexplem}\frac{s}{2s+1}<\alpha<\frac{1}{2},\eeq 
and
\beq\label{gammaexplem}0<\gamma<\min\{4s(1-\alpha)-2s,\alpha(2s+1)-s\}.\eeq
Let $\tau_\ep$ be such that 
\beq\label{e^tauepexplem}\varrho_\ep e^{-\frac{\mu \tau_\ep}{\ep^{2s+1}}}=\ep^{2s+\gamma},\eeq i.e.,
\beqs \tau_\ep=\frac{\ep^{2s+1}}{\mu}\log\left(\varrho_\ep\ep^{-(2s+\gamma)}\right).\eeqs
Remark that from \eqref{varroepmainthm}, $$ \tau_\ep=o(1)\quad\text{as }\ep\to 0.$$
We compute
\beqs \begin{split}\ep h_t&=\frac{\ep^{2s+1}}{\beta}\sigma_t -\sum_{i=1\atop i\neq K,K+1}^{N}\zeta_i\dot{x}_iu'\left(\zeta_i\frac{x-x_i(t)}{\ep}\right)-\ep^{-2s} \varrho_\ep \mu e^{-\frac{\mu t}{\ep^{2s+1}}}\\&
=\ep^{-2s-1}K_\ep \varrho_\ep\mu  e^{-\frac{\mu t}{\ep^{2s+1}}}\sum_{i=1\atop i\neq K,K+1}^{N}u'\left(\zeta_i\frac{x-x_i(t)}{\ep}\right)-\ep^{-2s} \varrho_\ep \mu e^{-\frac{\mu t}{\ep^{2s+1}}}
+O(\ep^{2s+1}),\end{split}\eeqs
and 
\beqs \I h=\frac{\ep^{2s}}{\beta}\I\sigma+\ep^{-2s}\sum_{i=1\atop i\neq K,K+1}^{N}\I u \left(\zeta_i\frac{x-x_i(t)}{\ep}\right)=\sum_{i=1\atop i\neq K,K+1}^{N}\ep^{-2s} W'\left(u \left(\zeta_i\frac{x-x_i(t)}{\ep}\right)\right)
+O(\ep^{2s}).\eeqs
Then, using the periodicity of $W'$, we get
\beq\label{Ieph}\begin{split}& \ep h_t-\I h+\displaystyle\frac{1}{\ep^{2s}}W'(h)-\sigma\\&=\ep^{-2s-1}K_\ep \varrho_\ep\mu  e^{-\frac{\mu t}{\ep^{2s+1}}} 
\sum_{i=1\atop i\neq K,K+1}^{N}u'\left(\zeta_i\frac{x-x_i(t)}{\ep}\right)-\ep^{-2s}\varrho_\ep \mu e^{-\frac{\mu t}{\ep^{2s+1}}}\\&
+\ep^{-2s} W'\left(\frac{\ep^{2s}}{\beta}\sigma+\sum_{i=1\atop i\neq K,K+1}^{N}u \left(\zeta_i\frac{x-x_i(t)}{\ep}\right)+\varrho_\ep e^{-\frac{\mu t}{\ep^{2s+1}}}\right)
-\sum_{i=1\atop i\neq K,K+1}^{N}\ep^{-2s}W'\left(u \left(\zeta_i\frac{x-x_i(t)}{\ep}\right)\right)\\&
-\sigma+O(\ep^{2s}).
\end{split}\eeq

\noindent\emph{Case 1.}
Suppose that there exits $i_0$ such that $x$ is close to $x_{i_0}(t)$ more than $\ep^\alpha$:
 \beqs%\label{xclosexiexpdeclem}
  |x-x_{i_0}(t)|\leq \ep^\alpha\eeqs for fixed  $\alpha$ satisfying \eqref{alphaexplem}. 
  Then estimate \eqref{u'infinity} implies
 \beq\label{iuioexplemset} u' \left(\zeta_{i_0}\frac{x-x_{i_0}(t)}{\ep}\right)\geq c \ep^{(1-\alpha)(1+2s)}.\eeq For $i\neq i_0$,  we simply  have
 \beq\label{uineioexplemst} u' \left(\zeta_{i}\frac{x-x_{i}(t)}{\ep}\right)\geq 0.\eeq
 From \eqref{x(t)} and the fact that the $x_i(t)$'s are well separated at time $t=0$ by   \eqref{xi+1-xiboundmainthm1}, we infer that
 for $K_\ep$ such that $K_\ep\varrho_\ep=o(1)$ as $\ep\to 0$, the $x_i(t)$'s stay well separated for any $t\in(0,\tau_\ep)$. Therefore, if $x$ is close $x_{i_0}(t)$, then there exists $c>0$ independent of $\ep$, such that for any $i\neq i_0$,
 $$ |x-x_i(t)|\geq c.$$ This combined with \eqref{uinfinity} yields, for $i\neq i_0$,
 \beq\label{uinei0lesse2s} \left|\widetilde{u} \left(\zeta_i\frac{x-x_i(t)}{\ep}\right)\right|\leq C\ep^{2s},\eeq
 where here and in what follows, we denote by $C$ several constants independent of $\ep$ and by
 $$\widetilde{u}(x):=u(x)-H(x),$$ where $H$ is the Heaviside function. 
 Hence, from the Lipschitz regularity and the perio\-di\-ci\-ty of $W'$, we get
 \beqs \begin{split} 
&\ep^{-2s} W'\left(\frac{\ep^{2s}}{\beta}\sigma+\sum_{i=1\atop i\neq K,K+1}^{N}u \left(\zeta_i\frac{x-x_i(t)}{\ep}\right)+\varrho_\ep e^{-\frac{\mu t}{\ep^{2s+1}}}\right)
-\ep^{-2s}W'\left(u \left(\zeta_i\frac{x-x_{i_0}(t)}{\ep}\right)\right)\\&
=\ep^{-2s} W'\left(\frac{\ep^{2s}}{\beta}\sigma+u \left(\zeta_{i_0}\frac{x-x_{i_0}(t)}{\ep}\right)+
\sum_{i\neq i_0}\widetilde{u} \left(\zeta_i\frac{x-x_i(t)}{\ep}\right)+\varrho_\ep e^{-\frac{\mu t}{\ep^{2s+1}}}\right)\\&
-\ep^{-2s}W'\left(u \left(\zeta_i\frac{x-x_{i_0}(t)}{\ep}\right)\right)\\&
\geq-C \ep^{-2s}\varrho_\ep e^{-\frac{\mu t}{\ep^{2s+1}}}-C.
\end{split}\eeqs
Moreover, from \eqref{uinei0lesse2s}, the Lipschitz regularity of $W'$ and $W'(0)=0$, we infer that 
\beqs \sum_{i\neq i_0}\ep^{-2s}\left|W'\left(u \left(\zeta_i\frac{x-x_i(t)}{\ep}\right)\right)\right|\leq C.\eeqs
 Therefore,  from \eqref{Ieph}, using the previous estimates, \eqref{e^tauepexplem}, \eqref{iuioexplemset} and \eqref{uineioexplemst},  we get, for any $(t,x)\in(0,\tau_\ep)\times\R$, 
\beqs\begin{split} \ep h_t-\I h+\displaystyle\frac{1}{\ep^{2s}}W'(h)-\sigma\geq& \frac{K_\ep \varrho_\ep\mu }{\ep^{2s+1}}e^{-\frac{\mu t}{\ep^{2s+1}}}c\ep^{(1-\alpha)(1+2s)}
-\frac{\varrho_\ep\mu}{\ep^{2s}} e^{-\frac{\mu t}{\ep^{2s+1}}}- \frac{C\varrho_\ep}{\ep^{2s}}e^{-\frac{\mu t}{\ep^{2s+1}}}-C\\&
=\varrho_\ep e^{-\frac{\mu t}{\ep^{2s+1}}}(cK_\ep\mu\ep^{-\alpha(1+2s)}-\mu\ep^{-2s}-C\ep^{-2s}-C\varrho_\ep^{-1} e^{\frac{\mu t}{\ep^{2s+1}}})\\&
\geq \varrho_\ep e^{-\frac{\mu t}{\ep^{2s+1}}}(cK_\ep\mu\ep^{-\alpha(1+2s)}-C\ep^{-2s-\gamma})
\\&=0\end{split} \eeqs
 if
 \beq\label{kepmuchoice} K_\ep \mu=\frac{C}{c}\ep^{\alpha(2s+1)-2s-\gamma}.\eeq
 Remark that  since $\frac{\varrho_\ep}{\ep^s}=$o(1) as $\ep\to0$ by \eqref{varroepmainthm}, for fixed $\mu$ independent of $\ep$, we have 
 $$\varrho_\ep K_\ep =o(1)\ep^{\alpha(2s+1)-s-\gamma}=o(1)\quad\text{as }\ep\to0, $$
 for $\gamma$ satisfying \eqref{gammaexplem}.
 
 \bigskip
 
 \noindent\emph{Case 2.} Suppose that, for any $i=1,\ldots,N-2$, 
 \beqs |x-x_i(t)|\geq \ep^\alpha.\eeqs
 Then,   estimate \eqref{uinfinity} implies 
  \beq\label{uinei0lesse2scase2} \left|\widetilde{u} \left(\zeta_i\frac{x-x_i(t)}{\ep}\right)\right|\leq C\ep^{2s(1-\alpha)}.\eeq
  Making a Taylor expansion of $W'$ around 0, using that $W'(0)=0$, $W''(0)=\beta>0$ and \eqref{uinei0lesse2scase2}, we get 
  \beqs\begin{split}
  &\ep^{-2s} W'\left(\frac{\ep^{2s}}{\beta}\sigma+\sum_{i=1\atop i\neq K,K+1}^{N}u \left(\zeta_i\frac{x-x_i(t)}{\ep}\right)+\varrho_\ep e^{-\frac{\mu t}{\ep^{2s+1}}}\right)\\&
  =\ep^{-2s} W'\left(\frac{\ep^{2s}}{\beta}\sigma+\sum_{i=1\atop i\neq K,K+1}^{N}\widetilde{u} \left(\zeta_i\frac{x-x_i(t)}{\ep}\right)+\varrho_\ep e^{-\frac{\mu t}{\ep^{2s+1}}}\right)\\&
  =\beta\ep^{-2s}\left(\frac{\ep^{2s}}{\beta}\sigma+\sum_{i=1\atop i\neq K,K+1}^{N}\widetilde{u}  \left(\zeta_i\frac{x-x_i(t)}{\ep}\right)+\varrho_\ep e^{-\frac{\mu t}{\ep^{2s+1}}}\right)+ \ep^{-2s}O(\ep^{2s(1-\alpha)})^2\\&
  +\ep^{-2s}O\left(\varrho_\ep e^{-\frac{\mu t}{\ep^{2s+1}}}\right)^2\\&
  \geq\sigma +\beta\ep^{-2s}\sum_{i=1\atop i\neq K,K+1}^{N}\widetilde{u}  \left(\zeta_i\frac{x-x_i(t)}{\ep}\right)+\frac{\beta}{2}\ep^{-2s}\varrho_\ep e^{-\frac{\mu t}{\ep^{2s+1}}}
  +O(\ep^{4s(1-\alpha)-2s}),
 \end{split} \eeqs
 for $\ep$ small enough. Similarly, we have
   \beqs\begin{split}
\sum_{i=1\atop i\neq K,K+1}^{N}\ep^{-2s}W'\left(u \left(\zeta_i\frac{x-x_i(t)}{\ep}\right)\right)=\beta\ep^{-2s}\sum_{i=1\atop i\neq K,K+1}^{N}\widetilde{u}  \left(\zeta_i\frac{x-x_i(t)}{\ep}\right)
 +O(\ep^{4s(1-\alpha)-2s}).
 \end{split} \eeqs
 
 Combining the previous estimates with \eqref{Ieph} and using that $u'>0$,  \eqref{gammaexplem} and \eqref{e^tauepexplem}, yields, for any $(t,x)\in\R\times (0,\tau_\ep)$,
  \beqs\begin{split}
 \ep h_t-\I h+\displaystyle\frac{1}{\ep^{2s}}W'(h)-\sigma&\geq 
 \frac{\beta}{2}\ep^{-2s}\varrho_\ep e^{-\frac{\mu t}{\ep^{2s+1}}}-\ep^{-2s}\varrho_\ep \mu e^{-\frac{\mu t}{\ep^{2s+1}}}+O(\ep^{4s(1-\alpha)-2s})\\&
 =\ep^{-2s}\varrho_\ep e^{-\frac{\mu t}{\ep^{2s+1}}}\left( \frac{\beta}{2}-\mu+O(\ep^{4s(1-\alpha)})\varrho_\ep^{-1} e^{\frac{\mu t}{\ep^{2s+1}}}\right)\\&
 \geq \ep^{-2s}\varrho_\ep e^{-\frac{\mu t}{\ep^{2s+1}}}\left( \frac{\beta}{2}-\mu-C\ep^{4s(1-\alpha)-2s-\gamma}\right)\\&
 = \ep^{-2s}\varrho_\ep e^{-\frac{\mu t}{\ep^{2s+1}}}\left( \frac{\beta}{2}-\mu-o(1)\right)\\&
 \geq 0,
 \end{split} \eeqs
   if we fix $\mu$ independent of $\ep$ such that
   \beq\label{muchoice}\mu\le \frac{\beta}{4},\eeq and $\ep$ is small enough.
 The lemma is then proven  choosing $\tau_\ep$, $K_\ep$ and $\mu$ satisfying  respectively  \eqref{e^tauepexplem}, \eqref{kepmuchoice} and \eqref{muchoice},   with  $\alpha$ and 
 $\gamma$ satisfying respectively \eqref{alphaexplem} and \eqref{gammaexplem}. 
   \end{proof}
Let us now conclude the proof of Theorem \ref{thmexponentialdecayNeven}.
From Theorem \ref{mainthmbeforecollNeven} we have
$$v_\ep(T_\ep^1,x)\leq h(0,x)\quad\text{for any } x\in\R.$$
Moreover, for $\mu$, $K_\ep$ and $\tau_\ep$ given by Lemma \ref{layerexponelem} and $\ep$ small enough, the function $h(t,x)$ is a
supersolution of the equation
\eqref{vepeq}. The comparison principle then implies
$$v_\ep(T_\ep^1+t,x)\leq h(t,x)\quad\text{for any } (t,x)\in(0,\tau_\ep]\times\R.$$
Choosing $t=\tau_\ep$ above, we get  \eqref{vlesepN-2exponthm} with $$\widetilde{\varrho}_\ep:=\varrho_\ep e^{-\frac{\mu \tau_\ep}{\ep^{2s+1}}}$$ satisfying \eqref{tauepvarromexpthm}. 

Finally, \eqref{xi-tildexiexpthm} is a consequence of \eqref{x(t)} and \eqref{varrotildelem}.

%%%%%%%%%%%%%%%%%%%%%%%%%%%%%%%%%%%%%
%%%%%%%%%%%%%%%%%%%%%%%%%%%%%%%%%%%%
\section{Proof of Theorems~\ref{Ncollisionevencor} 
and \ref{NcollisionN=2K+lcor}}\label{8jAJJa}

We perform a unique proof of Theorems~\ref{Ncollisionevencor} and \ref{NcollisionN=2K+lcor}.
Let $N=2K-l$, with either $l=0$ (Theorem~\ref{Ncollisionevencor})
or $0\ne l\in\N$ (Theorem~\ref{NcollisionN=2K+lcor}).   First of all, notice that, given $x_1^0,\ldots,x_N^0$, for  any $(\zeta_1,\ldots,\zeta_N)\in\{-1,1\}^N$ such that $\sum_{i=1}^N\zeta_i=l$, the initial datum $v_\ep^0$, defined in 
\eqref{vep03}, is below the function $w_\ep^0$ in which the positive particles are the first $K$  and the negative ones the remaining last $K-l$, i.e.,  for any $x\in\R$, 
\beqs v_\ep^0(x)\leq w_\ep^0(x):=\displaystyle\frac{\ep^{2s}}{\beta}
\sigma(0,x)+ \sum_{i=1}^{K}u\left(
\displaystyle\frac{x-x_i^0}{\ep}\right)+ \sum_{i=K+1}^{N}u\left(
\displaystyle\frac{x_i^0-x}{\ep}\right)-(N-K).
\eeqs
The comparison principle then implies, 
\beq\label{vepwepcomparison} v_\ep(t,x)\leq w_\ep(t,x)\quad\text{for any }(t,x)\in(0,+\infty)\times \R,\eeq
where $ w_\ep$ is the solution of  \eqref{vepeq} with initial datum $ w_\ep^0$. Therefore, when $l=0$,  to show that there exist $\mathcal{T}_\ep^K$ and $\Lambda_\ep^K=o(1)$ as $\ep\to 0$ such that 
\beq\label{vepabovecor} v_\ep(\mathcal{T}_\ep^K,x)\leq \Lambda_\ep^K\quad \text{for any }x\in\R,\eeq
it suffices to prove 
\eqref{vepabovecor} for $w_\ep(t,x)$. When $l\in\N$ it suffices to show \eqref{vepNcollN=2K+lcorabove} for $w_\ep(t,x)$.

Hence, let us consider the solution $(x_1(t),\ldots,x_N(t))$  of the 
ODE's system \eqref{dynamicalsysNeven} with 
$$\zeta_i=\begin{cases}1& \text{for }i=1,\ldots,K\\ -1& \text{for }i=K+1,\ldots,N.
\end{cases}$$
As usual, let us denote, for $i=1,\ldots, N-1$,  $$\theta_i(t):=x_{i+1}(t)-x_i(t)$$ and 
$$\theta_i^0:=x_{i+1}^0-x_{i}^0.$$

Let us first assume $\sigma\equiv0$. 
From Lemma \ref{T^1_cboundlemma}, for any initial configuration of the particles, a collision between the particles $x_K$ and $x_{K+1}$ of system \eqref{dynamicalsysNeven} occurs at a finite  time, that we denote by $T_c^1$, satisfying 
\beqs T_c^1\leq\frac{s(\theta_{K}^0)^{1+2s}}{(2s+1)\gamma}.\eeqs
Then by Theorems \ref{mainthmbeforecollNeven} and \ref{thmexponentialdecayNeven}, there exist $T^1_\ep, o^1_\ep>0$ and  
$\widetilde{x}_1^{1,\ep},\ldots,\widetilde{x}_{K-1}^{1,\ep},\widetilde{x}_{K+2}^{1,\ep}\ldots,\widetilde{x}_{N}^{1,\ep}$,  such that, for $i\in\{1,\ldots,K-1,K+2,\ldots,N\}$, 
\beqs\widetilde{x}_i^{1,\ep}=x_i(T_c)+o(1) \quad \text{as }\ep\to0%\begin{cases} x_i(T_c)+o(1)&\text{for }i=1,\ldots,K-1\\
%x_{i+2}(T_c)+o(1)&\text{for }i=K,\ldots,N-2,
%\end{cases} 
\eeqs
\beqs T_\ep^1=T_c^1+o(1),\quad 0<o_\ep^1:=\beta\frac{\tilde{\varrho}_\ep}{\ep^{2s}}=o(1)
 \quad\text{as }\ep\to 0,\eeqs
and 
 \beq\label{wlesepN-2exponcor} w_\ep(T_\ep^1,x)\leq \displaystyle\frac{\ep^{2s}}{\beta}o^1_\ep
 %\sigma(T_\ep^1,x)
 + \sum_{i=1}^{K-1}u\left(
\displaystyle\frac{x-\widetilde{x}_i^{1,\ep}}{\ep}\right)+ \sum_{i=K+2}^{N}u\left(
\displaystyle\frac{\widetilde{x}_i^{1,\ep}-x}{\ep}\right)-(N-K-1).\eeq
Now, let us denote by $w^1_\ep(t,x)$ the solution of system  \eqref{vepeq}, with $\sigma=o^1_\ep$ and initial datum the right-hand side of \eqref{wlesepN-2exponcor}. Then, from the comparison principle, we have, for any $(t,x)\in (0,+\infty)\times\R$,
\beq\label{wepw1epcomparison} w_\ep(T_\ep^1+t,x)\leq w^1_\ep(t,x).\eeq
From Lemma \ref{T^1_cboundlemma}, for $\ep$ small enough, the collision time, that we denote by $T_c^2$, of the following ODE's system: for $i\in\{1,\ldots,K-1,K+2,\ldots,N\}$,
\beq\label{ODEsystemsecondstepNeven}\begin{cases}\dot{x}_i=\gamma\left(
\displaystyle\sum_{j\neq i}\zeta_i\zeta_j 
\displaystyle\frac{x_i-x_j}{2s |x_i-x_j|^{1+2s}}+o_\ep^1\right)&\text{in }(0,T_c^2)\\
 x_i(0)=\widetilde{x}_i^{1,\ep},
\end{cases}\eeq
where 
\beqs\zeta_i=\begin{cases}1&\text{for }i=1,\ldots,K-1\\
-1&\text{for }i=K+2,\ldots,N,\\
\end{cases}
\eeqs
is finite. Therefore, by  Theorems  \ref{mainthmbeforecollNeven} and \ref{thmexponentialdecayNeven} and \eqref{wepw1epcomparison} , 
there exist $T^2_\ep,o^2_\ep>0$ and  
$\widetilde{x}_1^{2,\ep},\ldots,\widetilde{x}_{K-2}^{2,\ep},\widetilde{x}_{K+3}^{2,\ep}\ldots,\widetilde{x}_{N}^{2,\ep}$,  such that, for $i\in\{1,\ldots,K-2,K+3,\ldots,N\}$, we have
\beqs\widetilde{x}_i^{2,\ep}=x_i(T^2_c)+o(1) \quad \text{as }\ep\to0%\begin{cases} x_i(T_c)+o(1)&\text{for }i=1,\ldots,K-1\\
%x_{i+2}(T_c)+o(1)&\text{for }i=K,\ldots,N-2,
%\end{cases} 
\eeqs
where $(x_1,\ldots,x_{K-2},x_{K+3},\ldots,x_N)$ is the solution of \eqref{ODEsystemsecondstepNeven},
\beqs T_\ep^2=T_c^2+o(1),\quad o^2_\ep=o(1)
 \quad\text{as }\ep\to 0,\eeqs
and 
 \beqs w_\ep(T_\ep^1+T_\ep^2,x)\leq \displaystyle\frac{\ep^{2s}}{\beta}(o^1_\ep+o^2_\ep)
 %\sigma(T_\ep^1,x)
 + \sum_{i=1}^{K-2}u\left(
\displaystyle\frac{x-\widetilde{x}_i^{2,\ep}}{\ep}\right)+ \sum_{i=K+3}^{N}u\left(
\displaystyle\frac{\widetilde{x}_i^{2,\ep}-x}{\ep}\right)-(N-K-2).\eeqs
Let us first assume $l=0$. Then,
repeating the argument, we see that, after $K$ collisions, if we denote  $$\mathcal{T}_\ep^{K}:= T_\ep^1+\ldots +T_\ep^{K}$$
and $$\Lambda_\ep^K:= \displaystyle\frac{\ep^{2s}}{\beta}(o^1_\ep+\ldots+o^{K-1}_\ep)+\varrho_\ep^K,$$ then, for any $x\in\R$, 
\beqs w_\ep(\mathcal{T}_\ep^K,x)\leq \Lambda_\ep^K.\eeqs
The last estimate and \eqref{vepwepcomparison}  imply \eqref{vepabovecor}. Remark that Theorem  \ref{thmexponentialdecayNeven}  cannot be applied after the last collision, since
 there are only two  remaining particles before the last collision occurs, therefore the hypothesis   $N>2$ of the theorem is not satisfied. 

Similarly, when $l\in\N$, after $K-l$ collisions, if we denote  $$\mathcal{T}_\ep^{K-l}:= T_\ep^1+\ldots +T_\ep^{K-l}$$
and $$\Lambda_\ep^{K-l}:= \displaystyle\frac{\ep^{2s}}{\beta}(o^1_\ep+\ldots+o^{K-l}_\ep),$$ we get that $w_\ep(t,x)$, and therefore by  \eqref{vepwepcomparison} $v_\ep(t,x)$, satisfies inequality \eqref{vepNcollN=2K+lcorabove},
%\beq\label{wepabovecor} w_\ep(t,x)\leq \Lambda_\ep^{K-l},\eeq 
with $\Lambda_\ep^{K-l}$ satisfying \eqref{lambdak-lcor}. Differently from the previous case, when $l\in\N$, Theorem  \ref{thmexponentialdecayNeven}  can be applied after the last collision, since
 there are more than two  remaining particles before the last collision occurs.
 %Inequalities \eqref{wepabovecor} and    \eqref{vepwepcomparison} finally give \eqref{vepNcollN=2K+lcorabove}. 
To show \eqref{vepNcollN=2K+lcorbelow} when $l\in\N$ and    
\beq\label{vepbelowcor} v_\ep(\mathcal{T}_\ep^K,x)\geq -\Lambda_\ep^K\quad \text{for any }x\in\R,\eeq when $l=0$, 
we consider the function $z_\ep$ to be the solution of  \eqref{vepeq} with initial datum $z_\ep^0$ in which the negative particles are now the first $K-l$  and the positive ones the remaining last $K$, i.e., 
\beqs z_\ep^0(x):=\displaystyle\frac{\ep^{2s}}{\beta}
\sigma(0,x)+ \sum_{i=1}^{K-l}u\left(
\displaystyle\frac{x_i^0-x}{\ep}\right)+ \sum_{i=K-l+1}^{N}u\left(
\displaystyle\frac{x-x_i^0}{\ep}\right)-(N-K).
\eeqs
The comparison principle then implies 
\beqs v_\ep(t,x)\geq z_\ep(t,x)\quad\text{for any }(t,x)\in(0,+\infty)\times \R.\eeqs
A similar argument as before, then gives \eqref{vepNcollN=2K+lcorbelow} when $l\in\N$ and \eqref{vepbelowcor} when $l=0$. This concludes the proof of 
Theorems~\ref{Ncollisionevencor}  and \ref{NcollisionN=2K+lcor} in the case $\sigma\equiv 0$.

The result for $\sigma\not\equiv 0$ such that $\|\sigma\|_\infty\leq\overline{\sigma}$ with  $\overline{\sigma}$ small enough, follows from the case $\sigma\equiv 0$ and the continuity up to the collision time,  of the solution of the ODE's system 
\beqs\begin{cases}\dot{x}_i=\gamma\left(
\displaystyle\sum_{j\neq i}\zeta_i\zeta_j 
\displaystyle\frac{x_i-x_j}{2s |x_i-x_j|^{1+2s}}-\zeta_i\delta\right)&\text{in }(0,T_c)\\
 x_i(0)=\widetilde{x}_i^\ep,
\end{cases}\eeqs

with respect to the parameter $\delta$ (Proposition~\ref{tclimpropNeven}).

%%%%%%%%%%%%%%%%%%%%%%%%%%%%%%%%%%%%%%%%%%%
%%%%%%%%%%%%%%%%%%%%%%%%%%%%%%%%%%%%%%%%%%%%

\section{Proof of Theorems~\ref{Ncollisionevencorsigma0} 
and \ref{NcollisionN=2K+lcorinfinity}}\label{AKAaKA}

\subsection{Proof of Theorem~\ref{Ncollisionevencorsigma0}}
The proof of Theorem~\ref{Ncollisionevencorsigma0} follows the same steps as in the proof of  Theorem~1.2 in \cite{pv3} and we only sketch it. 
Consider the function $h(\tau,\xi)$ which is solution of
\beq\label{hode}\begin{cases}h_\tau+W'(h)=0,&\forall\tau\in(0,+\infty)\\
h(0,\xi)=\xi.
\end{cases}
\eeq
Then assumptions \eqref{Wass} and \eqref{lambdaK} imply that 
there exists $\ep_0>0$ such that for any $\ep<\ep_0$,
$h$  satisfies: $h(\tau,0)\equiv 0$;  if $\xi\in (0,\Lambda^K_\ep]$, then
\beqs0<h(\tau,\xi)\leq \xi e^{-\frac{\beta}{2}\tau};\eeqs
if $\xi\in [-\Lambda^K_\ep,0)$, then 
\beqs-\xi e^{-\frac{\beta}{2}\tau}\le h(\tau,\xi)<0,\eeqs
where $\beta=W''(0)>0$. 
Now, the function $\tilde{h}(t,x):=h(\frac{t-\mathcal{T}_\ep^K}{\ep^{2s+1}},\Lambda^K_\ep)$, where $\mathcal{T}_\ep^K$ is given by Theorem~\ref{Ncollisionevencor}, is solution of the equation \eqref{vepeq} for $\sigma\equiv0$ and $t>\mathcal{T}_\ep^K$, with  $\tilde{h}(\mathcal{T}_\ep^K,x)=\Lambda^K_\ep$. Then, the comparison principle and estimate \eqref{vepNcoll} imply 
$$v_\ep(t,x)\leq \tilde{h}(t,x)\leq \Lambda^K_\ep e^{-\frac{\beta}{2}\frac{t-\mathcal{T}_\ep^K}{\ep^{2s+1}}}\quad \text{for any }x\in\R,\,t>\mathcal{T}_\ep^K.$$ 
Similarly, one can  prove that
$$v_\ep(t,x)\geq -\Lambda^K_\ep e^{-\frac{\beta}{2}\frac{t-\mathcal{T}_\ep^K}{\ep^{2s+1}}}\quad \text{for any }x\in\R,\,t>\mathcal{T}_\ep^K,$$ and this proves \eqref{vexpontozeroNeven}.

%Theorem~\ref{Ncollisionevencorsigma0} is an immediate consequence of
%Theorem~\ref{Ncollisionevencor} here and Theorem 1.2 in \cite{pv3}.

\subsection{Proof of Theorem~\ref{NcollisionN=2K+lcorinfinity}}

We start by proving  a general result for the solution of the following system of ODE's: 
\beq\label{dynamicalsysNsameorient}\begin{cases}\dot{x}_i=\gamma
\displaystyle\sum_{j\neq i}
\displaystyle\frac{x_i-x_j}{2s |x_i-x_j|^{1+2s}}-\delta'(t)&\text{in }(0,T_c)\\
 x_i(0)=x_i^0,
\end{cases}\eeq
$i=1,\ldots,N$, where $\delta$ is a differentiable function. 
\begin{lem}\label{Nevenxm+1-xmlemsameor} Let $\delta:[0,+\infty)\to\R$ be  differentiable in $(0,+\infty)$. 
Let $(x_1(t),\ldots,x_N(t))$ be the solution of \eqref{dynamicalsysNsameorient} with  $x_{i+1}^0-x^0_{i}=\theta_0>0$, for any $i=1,\ldots, N-1$. Then there exists a constant $k$
depending on $N,\,\gamma, s$ and $\theta_0$, 
such that for any $i=1,\ldots,N-1$, we have 
\beq\label{Nevenm+1-xmlemsameoreq}x_{i+1}(t)-x_i(t)\geq k(1+t)^\frac{1}{1+2s}\quad\text{for any } t>0.\eeq
Moreover,  if  $N=2m$, $m\in\N$,  then
\beq\label{Nevenm+1-xmlemsameoreqbis}x_{m+1}(t)+x_m(t)=
x_{m+1}^0+x_m^0+2\delta(0)-2\delta(t)\quad\text{for any }t>0,\eeq
if instead, $N=2m+1$, $m\in\N$, then 
\beq\label{Nevenm+1-xmlemsameoreqbisodd}x_{m+1}(t)=x^0_{m+1}+\delta(0)-\delta(t)\quad\text{for any }t>0.\eeq
In particular $T_c=+\infty$. 
\end{lem}
\begin{proof}

We perform the proof of the lemma in the case $N=2m$, being the case $N=2m+1$ similar. 
Let us first consider the case $\delta\equiv 0$. 
Since the system of ODE's in  \eqref{dynamicalsysNsameorient} is invariant under translations of particles, that is, $(x_1(t)+a,\ldots,x_N(t)+a)$ is solution of the ODE's in  \eqref{dynamicalsysNsameorient}, for any $a\in\R$, without loss of generality we may assume that the initial configuration of the particles is symmetric with respect to the origin.
Therefore, suppose that, for $i=1,\ldots, m$, 
$$x^0_{m+i}=-x^0_{m-i+1}.$$
Then, the solution of   \eqref{dynamicalsysNsameorient} satisfies, for $i=1,\ldots, m$, 
\beq\label{symmcond}x_{m+i}(t)=-x(t)_{m-i+1},\quad\text{for any }t>0.\eeq
Indeed, let $(y_{m+1}(t),\ldots,y_{2m}(t))$ be the solution of the following system: for $i=1,\ldots, m$

\beqs\begin{cases}\dot{y}_{m+i}=\gamma\left(
\displaystyle\sum_{j=1\atop j\neq i }^m
\displaystyle\frac{y_{m+i}-y_{m+j}}{2s |y_{m+i}-y_{m+j}|^{1+2s}}+\displaystyle\sum_{ j=1}^m
\displaystyle\frac{y_{m+i}+y_{m+j}}{2s |y_{m+i}+y_{m+j}|^{1+2s}}\right)&\text{in }(0,T_c)\\
 y_{m+i}(0)=x_{m+i}^0.
\end{cases}\eeqs
Set, for $i=1,\ldots,m$ and $t\ge0$, 
$$y_{m-i+1}(t):=-y_{m+i}(t).$$
Then $(y_1(t),\ldots,y_N(t))$ is solution of \eqref{dynamicalsysNsameorient} and by uniqueness it coincides with $(x_1(t),\ldots,x_N(t))$. 
This implies that $(x_1(t),\ldots,x_N(t))$ satisfies property \eqref{symmcond}. In particular \eqref{Nevenm+1-xmlemsameoreqbis} holds true. 
Next, denote $$\theta_{j,i}(t):=x_{j}(t)-x_i(t).$$
In order to prove \eqref{Nevenm+1-xmlemsameoreq}, we show that for $j=1,\ldots,m$, there exists $k_j>0$ such that 
\beq\label{distancepartci}
\theta_{2m-j+1,j}(t)\ge k_j(1+ t)^\frac{1}{1+2s}.
\eeq
We prove \eqref{distancepartci} by induction. Let $j=1$. From \eqref{dynamicalsysNsameorient}, we see that $\theta_{2m,1}(t)$ solves:
\beqs\begin{split}
\dot{\theta}_{2m,1}&=\gamma\left(
\displaystyle\sum_{j=1 }^{2m-1}
\displaystyle\frac{1}{2s (x_{2m}-x_j)^{2s}}+\sum_{j=2 }^{2m}
\displaystyle\frac{1}{2s (x_{j}-x_1)^{2s}}\right)
\geq 
\frac{\gamma}{s \theta^{2s}_{2m,1}}.
\end{split}
\eeqs
A solution of equation $\dot{\theta}=\frac{\gamma}{s \theta^{2s}}$ is given by $\theta(t)=\left((N-1)^{1+2s}\theta_0^{1+2s}+\frac{(2s+1)\gamma}{s}t\right)^\frac{1}{2s+1}$. Since in addition, 
$\theta(0)=(N-1)\theta_0=\theta_{2m,1}(0)$, by comparison $\theta_{2m,1}(t)\geq \theta(t)$ for any $t>0$. This implies
\eqref{distancepartci} for $j=1$, with  $k_1= \min\left\{(N-1)\theta_0,\left(\frac{(2s+1)\gamma}{s}\right)^\frac{1}{2s+1}\right\}$.

Now assume that \eqref{distancepartci} holds true for $j=1,\ldots,m-1$ and let us prove it for $j=m$. Remark that, from  \eqref{symmcond}, we have, for $j=1,\ldots,m$,
\beqs\begin{split} \theta_{2m-j+1,j}&=x_{2m-j+1}-x_j=x_{2m-j+1}-x_{m+1}+\theta_{m+1,m}+x_m-x_j\\&
=2(x_{2m-j+1}-x_{m+1})+\theta_{m+1,m}=2(x_m-x_j)+\theta_{m+1,m}.
\end{split}
\eeqs
Therefore, from \eqref{dynamicalsysNsameorient}, we see that $\theta_{m+1,m}(t)$ solves:
\beqs\begin{split}
\dot{\theta}_{m+1,m}&=\frac{\gamma}{2s}\left(\displaystyle\sum_{j=1 }^{m}
\displaystyle\frac{1}{ (x_{m+1}-x_j)^{2s}}-\displaystyle\sum_{j=m+2 }^{2m}\displaystyle\frac{1}{ (x_j-x_{m+1})^{2s}}\right.\\&
\left.-\displaystyle\sum_{j=1 }^{m-1}
\displaystyle\frac{1}{ (x_{m}-x_j)^{2s}}+\displaystyle\sum_{j=m+1 }^{2m}\displaystyle\frac{1}{ (x_j-x_{m})^{2s}}\right)\\&
\ge \frac{\gamma}{2s}\left(\frac{2}{\theta^{2s}_{m+1,m}}-\displaystyle\sum_{j=1 }^{m-1}\displaystyle\frac{1}{ (x_{2m-j+1}-x_{m+1})^{2s}}-\displaystyle\sum_{j=1 }^{m-1}
\displaystyle\frac{1}{ (x_{m}-x_j)^{2s}}\right)\\&
=\frac{\gamma}{s}\left(\frac{1}{\theta^{2s}_{m+1,m}}-\displaystyle\sum_{j=1 }^{m-1}
\displaystyle\frac{2^{2s}}{ (\theta_{2m-j+1,j}-\theta_{m+1,m})^{2s}}\right).
%\\&\geq \frac{\gamma}{s}\left(\frac{1}{\theta^{2s}_{m+1,m}}-\displaystyle\sum_{j=1 }^{m-1}
%\displaystyle\frac{2^{2s}}{ (\theta_{2m-j+1,j})^{2s}}\right).
\end{split}
\eeqs
Then, using  \eqref{distancepartci} for $j=1,\ldots,m-1$, from the previous inequalities we get
 \beqs
 \dot{\theta}_{m+1,m}\ge \frac{\gamma}{s}\left(\frac{1}{\theta^{2s}_{m+1,m}}-\displaystyle\sum_{j=1 }^{m-1}
\displaystyle\frac{2^{2s}}{  (k_j (1+t)^\frac{1}{1+2s}-\theta^{2s}_{m+1,m})^{2s}}\right).
\eeqs
 Now, we consider the function $g(t)=k(1+t)^\frac{1}{1+2s}$ for some $0<k<k_j$ to be determined. We have
 \beqs \begin{split}
  &\dot{g}- \frac{\gamma}{s}\left(\frac{1}{g^{2s}}-\displaystyle\sum_{j=1 }^{m-1}
\displaystyle\frac{2^{2s}}{  (k_j(1+ t)^\frac{1}{1+2s}-g)^{2s}}\right)\\&=(1+t)^{-\frac{2s}{1+2s}}\left(\frac{k}{1+2s}- \frac{\gamma}{s}\left(k^{-2s}-\displaystyle\sum_{j=1 }^{m-1}
\displaystyle 2^{2s}( k_j-k)^{-2s}\right)\right)\leq 0,
\end{split}
\eeqs
for $k>0$ small enough. Therefore, there exists $k>0$ such that $g$ is subsolution of the equation
\beqs
 \dot{\theta}= \frac{\gamma}{s}\left(\frac{1}{\theta^{2s}}-\displaystyle\sum_{j=1 }^{m-1}
\displaystyle\frac{2^{2s}}{  (k_j(1+t)^\frac{1}{1+2s}-\theta)^{2s}} \right).
\eeqs
Since in addition, for $k\leq\theta_0$, we have that  $g(0)\leq\theta_{m+1,m}(0)$, by comparison we get $g(t)\leq\theta_{m+1,m}(t)$ for any $t>0$, i.e.,   \eqref{distancepartci} for $j=m$, with $k_j=k$.
This concludes the proof of \eqref{distancepartci}.
We are now ready to prove \eqref{Nevenm+1-xmlemsameoreq}.  From \eqref{symmcond} it suffices to show \eqref{Nevenm+1-xmlemsameoreq} for $i=m,\ldots,N-1.$
We proceed by induction. Inequality  \eqref{Nevenm+1-xmlemsameoreq} for $i=m$ is given by \eqref{distancepartci} for $j=m$. Assume now that 
\eqref{Nevenm+1-xmlemsameoreq} holds true for $i=m,\ldots,N-2.$ Then,  from \eqref{dynamicalsysNsameorient}, we see that $\theta_{N,N-1}(t)=x_N(t)-x_{N-1}(t)$ solves:
\beqs\begin{split}
\dot{\theta}_{N,N-1}&=\frac{\gamma}{2s}\left(\displaystyle\frac{2}{ \theta_{N,N-1}^{2s}}+\displaystyle\sum_{j=1 }^{N-2}
\displaystyle\frac{1}{ (x_{N}-x_j)^{2s}}-\displaystyle\sum_{j=1 }^{N-2}
\displaystyle\frac{1}{ (x_{N-1}-x_j)^{2s}}\right)\\&
\geq \frac{\gamma}{s}\left(\displaystyle\frac{1}{ \theta_{N,N-1}^{2s}}-\frac{C}{(1+t)^\frac{2s}{2s+1}}\right),\end{split}
\eeqs
 for some $C>0$. 
Arguing as before, we get \eqref{Nevenm+1-xmlemsameoreq} for $i=N-1$ and this concludes the proof of the lemma when $\delta\equiv 0$.
Now, let us consider the general case, when the assumption $\delta\equiv 0$ does not hold. Define $z_i(t):=x_i(t)+\delta(t)$, for $i=1,\ldots,N$. Then, 
$(z_1(t),\ldots,z_N(t))$ is solution of the initial value problem   \eqref{dynamicalsysNsameorient} with $\delta\equiv0$ and initial conditions $x_i^0+\delta(0)$. Therefore,  the results just proven in the case 
$\delta\equiv 0$ and applied to $(z_1(t),\ldots,z_N(t))$, yield  \eqref{Nevenm+1-xmlemsameoreq}, \eqref{Nevenm+1-xmlemsameoreqbis} and \eqref{Nevenm+1-xmlemsameoreqbisodd}
for $(x_1(t),\ldots,x_N(t))$. This concludes the proof of the lemma.
\end{proof}

Let us now prove Theorem~\ref{NcollisionN=2K+lcorinfinity}.   
In order to do it, we consider appropriate barriers for the solution $v_\ep$ of \eqref{vepeq}-\eqref{vep03}
with $\sigma\equiv0$. Set $$\theta_m:=\min_{i=1,\ldots,l-1}\underline{x}_{i+1}^\ep-\underline{x}_{i}^\ep$$ 
and $$0\le\sigma_\ep:=\frac{\Lambda_\ep^{K-l}}{\ep^{2s}}=o(1)\quad\text{as }\ep\to0,$$ 
where $\underline{x}_1^\ep,\ldots, \underline{x}_l^\ep$ and $\Lambda_\ep^{K-l}$
are given by 
Theorem~\ref{NcollisionN=2K+lcor}.
Let $w_\ep(t,x)$ be the solution of \eqref{vepeq} with $\sigma\equiv0$
 %\beq\label{sigmainfinitesimal}\sigma(t)=\sigma_\ep(1+t)^\frac{-2s}{2s+1}\eeq 
 and  with the following initial condition 
$$w_\ep(0,x)=\sum_{i=1}^{l}u\left(
\displaystyle\frac{x-\underline{y}_i^\ep}{\ep}\right)+\ep^{2s}\sigma_\ep,$$
where $u$ is the solution of \eqref{u}, and 
$\underline{y}_1^\ep,\ldots, \underline{y}_l^\ep$ are defined as follows
$$\underline{y}_1^\ep:=\underline{x}_1^\ep,\quad\quad \underline{y}_{i}^\ep:=\underline{x}_{i}^\ep+\theta_m,\text{ for }i=2,\ldots,l.$$
From \eqref{vepNcollN=2K+lcorabove} and the monotonicity of $u$, we have that $v_\ep(\mathcal{T}_\ep^{K-l},x)\leq w_\ep(0,x)$ for any $x\in\R$. Then by the comparison principle
\beq\label{vepl<wepinfintytime}v_\ep(\mathcal{T}_\ep^{K-l}+t,x)\leq w_\ep(t,x)\quad \text{for any }(t,x)\in(0,+\infty)\times\R.\eeq

\noindent Now, we argue as in Section \ref{IAL9}. 
Consider the function 
\beq\label{overwepsuper}\overline{w}_\ep(t,x) := \ep^{2s}\overline{\sigma}_\ep+\sum_{i=1}^l u\left(\displaystyle\frac{x-x_i(t)}{\ep}\right)
-\sum_{i=1}^l\ep^{2s}c_i(t)\psi\left(\displaystyle\frac{x-x_i(t)}{\ep}\right)
\eeq
where $u$ and $\psi$ are respectively the solution of \eqref{u} and \eqref{psi}, $(x_1(t),\ldots,x_l(t))$ is the solution of \eqref{dynamicalsysNsameorient} with 
\beq\label{wbarbarriercond2}N=l,\quad \delta(t)=(1+2s)(\sigma_\ep+\delta_\ep)(1+t)^{\frac{1}{1+2s}}\quad\text{and }x_i^0= \underline{y}_{i}^\ep-\delta_\ep,\eeq  and where 
\begin{equation}\label{wbarbarriercond3}\begin{split}&c_i(t)=\dot{x}_i(t),\\&\overline{\sigma}_\ep(t)=\frac{\delta'(t)}{W''(0)}=\frac{(\sigma_\ep+\delta_\ep)(1+t)^{-\frac{2s}{1+2s}}}{W''(0)}
%=\frac{\sigma(t)+\delta_\ep(t+1)^{-\frac{2s}{1+2s}}}{\beta}
,\end{split}\end{equation}
 and $\delta_\ep=o(1)$ as $\ep\to0$ to be determined. 
We want to show that there exists $\delta_\ep$ such that the function $\overline{w}_\ep(t,x)$ is an upper barrier for $w_\ep(t,x)$. By Lemma~\ref{initialcondpropNeven}, we have that 
\beq\label{initialcondwep}w_\ep(0,x)\leq \overline{w}_\ep(0,x)\quad\text{for any }x\in\R.\eeq
Moreover, $\overline{w}_\ep(t,x)$ is a supersolution of \eqref{vepeq}, as stated in the following proposition. 
\begin{prop}\label{thetaepropNevenlpart} There exist $\ep_0>0$ and $0<\delta_\ep=o(1)$ as $\ep\to0$, such that 
such that for any $\ep<\ep_0$, if $(x_1,\ldots,x_l)$ is a solution of the 
ODE system in \eqref{dynamicalsysNsameorient} where $N$ and  $\delta(t)$,   are given by \eqref{wbarbarriercond2}, then the  function $\overline{w}_\ep$ defined by \eqref{overwepsuper} satisfies
\beq\label{supersolwep}\ep(\overline{w}_\ep)_t-\I \overline{w}_\ep+\displaystyle\frac{1}{\ep^{2s}}W'(\overline{w}_\ep)\geq 0\eeq
%$$\lim_{\ep\to0}\theta_\ep=\lim_{\ep\to0}\delta_\ep=0$$ and 
for any $(t,x)\in(0,+\infty)\times\R$.
\end{prop}
Proposition \ref{thetaepropNevenlpart} generalizes Proposition~\ref{thetaepropNeven} in the  case in which the particles $x_i$'s have all the same orientation. 
Indeed, thanks to Lemma~\ref{Nevenxm+1-xmlemsameor},  in the former proposition the error term $\delta'$, appearing in system \eqref{dynamicalsysNsameorient}, is allowed to go to 0 as $t\to+\infty$. 
The proof of Proposition \ref{thetaepropNevenlpart} is a technical
modification of the proof of Proposition~\ref{thetaepropNeven} given in \cite{pv3}. Therefore, we postpone it to the Appendix.

Now, let us choose  $\delta_\ep$ such that \eqref{initialcondwep} and \eqref{supersolwep} hold. Then the comparison principle implies 
\beq\label{wepl<upsiepinfintytime} w_\ep(t,x)\leq \overline{w}_\ep(t,x)\quad\text{for any }(t,x)\in(0,+\infty)\times\R.
\eeq

Let us first consider the case $l=2m$. By Lemma~\ref{Nevenxm+1-xmlemsameor} applied with $\delta$ defined as in \eqref{wbarbarriercond2}, 
%$$\delta(t)= (1+2s)(\sigma_\ep+\delta_\ep)(t+1)^{\frac{1}{1+2s}},$$ 
we have that 
\beqs\begin{split}x_{m+1}(t)&=\frac{x_{m+1}(t)}{2}+\frac{x_{m+1}(t)}{2}=\frac{x_{m+1}^0+x_m^0}{2}+\delta(0)-\delta(t)+\frac{1}{2}(x_{m+1}(t)-x_{m}(t))\\&
\ge \frac{x_{m+1}^0+x_m^0}{2}+\left(\frac{k}{2}-(1+2s)(\sigma_\ep+\delta_\ep)\right)(1+t)^{\frac{1}{1+2s}}\\&
\ge \frac{x_{m+1}^0+x_m^0}{2}+\frac{k}{4}(1+t)^{\frac{1}{1+2s}},
\end{split}\eeqs
for $\ep$ small enough. Similarly, 
\beqs x_{m}(t)\leq \frac{x_{m+1}^0+x_m^0}{2}-\frac{k}{4}(t+1)^{\frac{1}{1+2s}},\eeqs
for $\ep$ small enough. 
From the previous estimates and \eqref{Nevenm+1-xmlemsameoreq}, we infer that,  for any $R>0$ there exists $t_0>0$ such that if  $|x|\le R$, we have, for any 
$t>t_0$,
\beqs x_m(t)<x<x_{m+1}(t),\quad\text{and }|x-x_i(t)|\ge C (1+t)^\frac{1}{2s+1},\quad\text{for any }i=1,\ldots,l.\eeqs 
Therefore, from \eqref{uinfinity}, we have
\beqs u\left(\displaystyle\frac{x-x_i(t)}{\ep}\right)\leq\begin{cases} 1+C\ep^{2s}(1+t)^{-\frac{2s}{2s+1}}&\text{if }i=1,\ldots,m\\
C\ep^{2s}(1+t)^{-\frac{2s}{2s+1}}&\text{if }i=m+1,\ldots,l.
\end{cases}
\eeqs
%This implies that 
%\beqs%\label{limituxi}
%\limsup_{t\to+\infty} \sum_{i=1}^l u\left(\displaystyle\frac{x-x_i(t)}{\ep}\right)\leq m.\eeqs
Next, from \eqref{dynamicalsysNsameorient} and \eqref{Nevenm+1-xmlemsameoreq}, we see that 
$$|c_i(t)|\leq C(1+t)^{-\frac{2s}{2s+1}}.$$
From the previous estimates, \eqref{vepl<wepinfintytime},  \eqref{overwepsuper} and \eqref{wepl<upsiepinfintytime}, we conclude that, for $t>t_0$,
\beq\label{vepestitinfaboveven}%\limsup_{t\to+\infty}v_\ep(t,x)\leq m+ \overline{\delta}_\ep.
v_\ep(\mathcal{T}_\ep^{K-l}+t,x)\leq m+C\ep^{2s}(1+t)^{-\frac{2s}{2s+1}}.
\eeq 
Similarly, choosing as lower barrier the function 
$z_\ep(t,x)$ solution of \eqref{vepeq} with $\sigma\equiv 0$ and initial condition 
$$z_\ep(0,x)=\sum_{i=1}^{l}u\left(
\displaystyle\frac{x-\overline{y}_i^\ep}{\ep}\right)-\ep^{2s}\sigma_\ep,$$ where 
$$\overline{y}_1^\ep:=\overline{x}_1^\ep,\quad\quad \overline{y}_{i}^\ep:=\overline{x}_{i}^\ep+\theta_M,\text{ for }i=2,\ldots,l,$$
$$\theta_M:=\max_{i=1,\ldots,l-1}\overline{x}_{i+1}^\ep-\overline{x}_{i}^\ep,$$
and $\overline{x}_1^\ep,\ldots, \overline{x}_l^\ep$ are given by 
Theorem~\ref{NcollisionN=2K+lcor}, we obtain, for $|x|<R$ and $t>t_0$,
\beq\label{vepestitinfbeloweven}%\liminf_{t\to+\infty}v_\ep(t,x)\geq m-\overline{\delta}_\ep.
v_\ep(\mathcal{T}_\ep^{K-l}+t,x)\geq m-C\ep^{2s}(1+t)^{-\frac{2s}{2s+1}}.
\eeq
%For fixed $\ep$ the function $v_\ep(t,x)$ is H{\"o}lder continuous in $x$ uniformly in time, see e.g. \cite{mp}. Then, there exists a sequence $(t_k)_k$ with 
%$t_k\to+\infty$ as $k\to \infty$ such that 
%$$v_\ep(t_k,x)\to v_\ep^\infty(x)\quad\text{as }k\to \infty,$$
%with $v_\ep^\infty(x)$ viscosity solution of the stationary equation 

%$$\I v=\displaystyle\frac{1}{\ep^{2s}}W'(v)\quad\text{in }\R.$$
%Moreover, from \eqref{vepestitinfaboveven} and \eqref{vepestitinfbeloweven}, we have that 
%\beqs | v_\ep^\infty(x)-m|\le\overline{\delta}_\ep\quad\text{for any }x\in\R.\eeqs
%We claim that for $\ep$ small enough, we have
%\beq\label{vinftym}v_\ep^\infty\equiv m.\eeq
 %Indeed, let $h(\tau,\xi)$ be the function defined by \eqref{hode}. As in the proof of Theorem~\ref{Ncollisionevencorsigma0}, there exists $\ep_0>0$ such that for any $\ep<\ep_0$ and %any  
% $\xi\in[-\overline{\delta}_\ep,\overline{\delta}_\ep]$, we have $|h(\tau,\xi)|\le \xi e^{-\frac{\beta}{2}\tau}$, where $\beta=W''(0)>0$. Since $W$ is periodic the function 
 %$$\tilde{h}(t,x):=m+h\left(\frac{t}{\ep^{2s+1}},\overline{\delta}_\ep\right)$$ is solution of \eqref{vepeq} with $\sigma\equiv 0$. The comparison principle then implies 
% $$|v_\ep^\infty(x)-m|\leq |\tilde{h}(t,x)-m|\le e^{-\frac{\beta t}{2\ep^{2s+1}}}.$$
 %Letting $t$ go to infinity we get \eqref{vinftym}. Since the limit function does not depend on the particular subsequence, we conclude that the whole sequence $v_\ep(t,x)$ converges %to $m$ as $t\to+\infty$.  This proves  Theorem~\ref{NcollisionN=2K+lcorinfinity} in the case $l=2m.$ 
 
 Estimates \eqref{vepestitinfaboveven} and \eqref{vepestitinfbeloweven} give \eqref{vinfinitycorleven}. 
 %imply that $v_\ep(t,x)\to m$ as $t\to+\infty$ for $\ep$ small enough. 
 %This proves  Theorem~\ref{NcollisionN=2K+lcorinfinity} in the case $l=2m.$
 
 Let us now turn to the case $l=2m+1$. Fix $R>0$ and let $x\in \R$ such that $|x|\leq R$. Then, as before, from \eqref{Nevenm+1-xmlemsameoreq} and \eqref{Nevenm+1-xmlemsameoreqbisodd}, we infer that 
 exist $t_0>0$ and a constant $C>0$ such that for any 
$t>t_0$,
\beqs x_m(t)<x<x_{m+2}(t),\quad|x-x_i(t)|\ge C (1+t)^\frac{1}{2s+1},\,i\neq m+1\eeqs and for any $t>0$
\beqs x_{m+1}(t)=\underline{y}_{m+1}^\ep+\delta(0)-\delta(t)=\underline{y}_{m+1}^\ep-(1+2s)(\sigma_\ep+\delta_\ep)[(1+t)^\frac{1}{1+2s}-1]=\underline{x}^\ep-\alpha_\ep[(1+t)^\frac{1}{1+2s}-1],\eeqs
where 
$$\underline{x}^\ep:=\underline{y}_{m+1}^\ep,$$ and
\beqs%\label{alphaep}
\alpha_\ep:=(1+2s)(\sigma_\ep+\delta_\ep)=o(1)\quad\text{as }\ep\to0.\eeqs
We remark that from Theorem~\ref{NcollisionN=2K+lcor}, $\underline{x}^\ep$ is bounded with respect to $\epsilon$.
Therefore, from \eqref{uinfinity}, we have
\beqs u\left(\displaystyle\frac{x-x_i(t)}{\ep}\right)\leq\begin{cases}1+C\ep^{2s}(1+t)^{-\frac{2s}{2s+1}}&\text{if }i=1,\ldots,m\\
C\ep^{2s}(1+t)^{-\frac{2s}{2s+1}}&\text{if }i=m+2,\ldots,l.
\end{cases}
\eeqs
Moreover 
\beqs u\left(\displaystyle\frac{x-x_{m+1}(t)}{\ep}\right)=u\left(\displaystyle\frac{x-\underline{x}^\ep+\alpha_\ep[(1+t)^\frac{1}{1+2s}-1]}{\ep}\right)\eeqs
%This implies that 
%\beqs%\label{limituxi}
%\limsup_{t\to+\infty} \sum_{i=1}^l u\left(\displaystyle\frac{x-x_i(t)}{\ep}\right)\leq m.\eeqs
 and from \eqref{dynamicalsysNsameorient} and \eqref{Nevenm+1-xmlemsameoreq}, we see that 
$$|c_i(t)|\leq C(1+t)^{-\frac{2s}{2s+1}}.$$
As before, the previous estimates, \eqref{vepl<wepinfintytime} and \eqref{wepl<upsiepinfintytime}, imply  \eqref{vinfinitycorlodd2}.
 Similarly one can prove \eqref{vinfinitycorlodd1}.
%$$v_\ep(t,x)\leq m+u\left(\displaystyle\frac{x-\underline{x}^\ep+\alpha_\ep(1+t)^\frac{1}{1+2s}}{\ep}\right)+C (1+t)^{-\frac{2s}{2s+1}}.$$
%Similarly, one can find $\overline{x}^\ep\ge \underline{x}^\ep$ such that 
%$$v_\ep(t,x)\geq m+u\left(\displaystyle\frac{x-\overline{x}^\ep-\alpha_\ep(1+t)^\frac{1}{1+2s}}{\ep}\right)-C (1+t)^{-\frac{2s}{2s+1}}.$$
This concludes the proof of Theorem~\ref{NcollisionN=2K+lcorinfinity}.  

\section{Proof of Corollary~\ref{nostationarypoints}}\label{FIF}

We argue by contradiction and suppose that there
exists a constant solution
\begin{equation}\label{AJJ:kos}
(x_1(t),\dots,x_N(t))=(x_1^0,\dots,x_N^0)\end{equation}
of~\eqref{dynamicalsysNeven} with~$\sigma\equiv0$
and~$N\ge2$.
Without loss of generality, we suppose that the number of the positive~$
\zeta_i$'s, $K$,  is larger or equal than the number of
the negative ones, $N-K$.

Let $R>0$ be such that $|x_i^0|,\,\,|\underline{x}_\ep|,\,|\overline{x}_\ep|<R$, for any $i=1,\ldots,N$ and $\ep>0$, where $\underline{x}_\ep$ and $\overline{x}_\ep$ are given by
Theorem~\ref{NcollisionN=2K+lcorinfinity}. Pick any point $p<\min\{x_1^0, \underline{x}_\ep\}$ with  $|p|<2R$. Then by \eqref{vinfinitycorleven}, \eqref{vinfinitycorlodd1} and \eqref{vinfinitycorlodd2}, there exists $T_0>0$ such that for any $t>T_0$, we have 
$$\lim_{\ep\to0}v_\ep(t,p)=m.$$ 
On the other hand, since  $p<x_1^0$, by Theorem~1.1 in~\cite{pv2}
and~\eqref{AJJ:kos}, we have that
$$ \lim_{\ep\to0} v_\ep(t,p)
=\sum_{i=1}^N H(\zeta_i(p-x_i^0))-(N-K)=0,$$
where~$H$ is the
Heaviside function. Therefore, we must have $m=0$. 

Next, we fix~$N+1$ points, say~$p_1,\dots,p_{N+1}$, with $|p_i|<2R$ for any $p=1,\ldots,N+1$, such that
\begin{equation}\label{p:ord}
p_1 < x_1^0< p_2 < x_2^0<\dots<x_N^0 
<p_{N+1}\end{equation}
%and none of the~$p_i$'s is a limit point for the (sub)sequences~$\underline{x}_\ep$, and $\overline{x}_\ep$.
and we denote~$P:=\{p_1,\dots,p_{N+1}\}$.
By Theorem~1.1 in~\cite{pv2}
and~\eqref{AJJ:kos}, we have that, for any $p\in P$, and $t>0$,
\begin{equation}\label{OBa8} \lim_{\ep\to0} v_\ep(t,p)
=\sum_{i=1}^N H(\zeta_i(p-x_i^0))-(N-K).\end{equation}

We remark that the right hand side of~\eqref{OBa8}
is the superposition of~$N$ 
Heaviside functions
(up to a vertical translation).  Accordingly,
the values taken by
the right hand side of~\eqref{OBa8} 
have  $N$ jumps of size~$1$
when~$p\in P$ (recall~\eqref{p:ord}).

On the other hand, when $l=K-(N-K)=0$, by \eqref{vinfinitycorleven}, for any $t>T_0$  and $p\in P$, we have 
$$ \lim_{\ep\to0} v_\ep(t,p)=0$$
which is a contradiction. 

When $l=1$, by \eqref{vinfinitycorlodd1} and \eqref{vinfinitycorlodd2}, we must have that, for any $p\in P$,
$$\sum_{i=1}^N H(\zeta_i(p-x_i^0))-(N-K)\in\{0,1\},$$ which means that the particles $(x_1^0,\dots,x_N^0)$ must have alternate orientation. This is in contradiction with Theorem 1.6 of 
\cite{pv2} which states that in the case of alternate dislocations, when $\sigma\equiv0$, for any initial configuration  there is always a collision in finite time, in particular system 
\eqref{dynamicalsysNeven} does not admit stationary solutions.

Corollary~\ref{nostationarypoints} is then proven.

\section*{Appendix. Proof of Proposition~\ref{thetaepropNevenlpart}}
In order to simplify the notation, 
we set, for $i=1,\ldots,N$
\beq\label{utilde3}\tilde{u}_i(t,x):=u\left(\displaystyle\zeta_i\frac{x-x_i(t)}{\ep}\right)-H\left(\displaystyle\zeta_i\frac{x-x_i(t)}{\ep}\right),\eeq
where $H$ is the Heaviside function
and 
\beqs \psi_i(t,x):=\psi\left(\displaystyle\zeta_i\frac{x-x_i(t)}{\ep}\right).\eeqs

\noindent  Finally, let
\beq\label{vepansbarbis3}I_\ep:=\ep(\overline{w}_\ep)_t+\displaystyle\frac{1}{\ep^{2s}}W'(\overline{w}_\ep)-\I \overline{w}_\ep.\eeq
We want to find $\delta_\ep$ such that $I_\ep\ge 0$. 
To do it,  we need the following result, which is proven in \cite{pv3}. 

\begin{lem}[Lemma 8.1 in \cite{pv3}] For any $(t,x)\in (0,+\infty)\times\R$ we have, for $i=1,\ldots,N$
\beq\label{ieplem3}\begin{split}I_\ep&=O(\tilde{u}_i)(\ep^{-2s}\displaystyle\sum_{j\neq i}\tilde{u}_j+\overline{\sigma}_\ep+c_i\eta)+\frac{\delta'}{\gamma}\\&
+\sum_{j=1}^N\left\{O(\ep^{2s+1}\dot{c}_j)+O(\ep^{2s}c_j^2)\right\}\\&
+\sum_{j\neq i}\left\{O(c_j\psi_j)+O(c_j\tilde{u}_j)+O(\ep^{-2s}\tilde{u}_j^2)\right\}+O(\ep^{2s}),
\end{split} 
\eeq
where $\eta$ and $\gamma$ are given respectively  by \eqref{eta} and \eqref{gamma}.
\end{lem}

Let us  proceed with the proof of Proposition~\ref{thetaepropNevenlpart}. We consider 
 two cases. 
\bigskip

\noindent\emph{Case 1.}
Suppose that  $x$ is close to $x_i(t)$ more than $\ep^\alpha$, for some $i=1,\ldots,N$:
\beq\label{x-xs2leqeppow3}|x-x_i(t)|\leq \ep^\alpha\quad\text{with }0<\alpha<1.%<\displaystyle\frac{\kappa-2s}{\kappa},
\eeq
 Then, from \eqref{Nevenm+1-xmlemsameoreq},  for $j\neq i$,
\beq\label{x-xsjcase1promtheepbis}|x-x_j(t)|\geq C(1+t)^\frac{1}{1+2s}.\eeq   
Here and in what follows we denote by $C>0$ several constants independent of $\ep$. 
Hence, from \eqref{uinfinity},   \eqref{utilde3} and \eqref{x-xsjcase1promtheepbis},  we get 
\beqs\begin{split}&\left|\displaystyle\frac{\tilde{u}_j(t,x)}{\ep^{2s}}+\displaystyle\frac{1}{2s W''(0)}\displaystyle\frac{x-x_j(t)}{|x-x_j(t)|^{1+2s}}\right|
\\&= \displaystyle\frac{1}{\ep^{2s}}\left|u\left(\displaystyle\frac{x-x_j(t)}{\ep}\right)-H\left(\displaystyle\frac{x-x_j(t)}{\ep}\right)+
\displaystyle\frac{\ep^{2s}}{2s W''(0)}\displaystyle\frac{x-x_j(t)}{|x-x_j(t)|^{1+2s}}\right|\\&
\leq C\displaystyle\frac{\ep^\kappa}{\ep^{2s}}\displaystyle\frac{1}{|x-x_j(t)|^\kappa}\\&\leq C\ep^{\kappa-2s}(1+t)^{-\frac{k}{1+2s}},\end{split}\eeqs
where $\kappa>2s$  is given in Lemma \ref{uinfinitylem}.
Next, a Taylor expansion of the function $\displaystyle\frac{x-x_j(t)}{|x-x_j(t)|^{1+2s}}$ around $x_i(t)$, gives
\beqs\begin{split}\left|\displaystyle\frac{x-x_j(t)}{|x-x_j(t)|^{1+2s}}-\displaystyle\frac{x_i(t)-x_j(t)}{|x_i(t)-x_j(t)|^{1+2s}}\right|&\leq \displaystyle\frac{2s}{|\xi-x_j(t)|^{1+2s}}|x-x_i(t)|\leq C\ep^  \alpha (1+t)^{-1},\end{split}\eeqs
where $\xi$ is a suitable point lying on the segment joining $x$ to $x_i(t)$. 
The last two inequalities imply for $j\neq i$
\beq\label{u2behavioinfty3}\left|\displaystyle\frac{\tilde{u}_j(t,x)}{\ep^{2s}}+\displaystyle\frac{1}{2s W''(0)}\displaystyle\frac{x_i(t)-x_j(t)}{|x_i(t)-x_j(t)|^{1+2s}}\right|\leq C(\ep^{\kappa-2s} (1+t)^{-\frac{k}{1+2s}}+\ep^{ \alpha}(1+t)^{-1}).\eeq

\noindent Therefore, from  \eqref{ieplem3}, we get that
\beq\label{iepsemifinal3}\begin{split} I_\ep&=O(\tilde{u}_i)\left(\displaystyle\sum_{j\neq i}-\displaystyle\frac{1}{2s W''(0)}\displaystyle\frac{x_i(t)-x_j(t)}{|x_i(t)-x_j(t)|^{1+2s}}+\overline{\sigma}_\ep+c_i\eta\right)+\frac{\delta' }{\gamma}\\&
+C\ep^{\kappa-2s} (1+t)^{-\frac{k}{1+2s}}+C\ep^\alpha(1+t)^{-1}\\&
+\sum_{j=1}^N\left\{O(\ep^{2s+1}\dot{c}_j)+O(\ep^{2s}c_j^2)\right\}\\&
+\sum_{j\neq i}\left\{O(c_j\psi_j)+O(c_j\tilde{u}_j)+O(\ep^{-2s}\tilde{u}_j^2)\right\}.
\end{split}\eeq
Now, from \eqref{wbarbarriercond3}, the definition of $\eta$ given in \eqref{eta} $(\eta=\frac{1}{\gamma W''(0)})$ and \eqref{dynamicalsysNsameorient}, we see that
\beq\label{parenttermesti3}\displaystyle\sum_{j\neq i}-\displaystyle\frac{1}{2s W''(0)}\displaystyle\frac{x_i(t)-x_j(t)}{|x_i(t)-x_j(t)|^{1+2s}}+\overline{\sigma}_\ep+c_i\eta=0.\eeq

\noindent Let us next estimate the remaining  terms in \eqref{iepsemifinal3}. From the definition of $c_i(t)$ given in \eqref{wbarbarriercond3}, system \eqref{dynamicalsysNsameorient} and estimates \eqref{Nevenm+1-xmlemsameoreq},  we have for $j=1,\ldots,N$
\beq\label{c1c2behavior3}|c_j|=O((1+t)^{-\frac{2s}{1+2s}}),\eeq then 
\beq\label{errorestimeateIep13}O(\ep^{2s}c_j^2)=O(\ep^{2s}(1+t)^{-\frac{4s}{1+2s}}).\eeq

\noindent Next, differentiating the equations in \eqref{dynamicalsysNsameorient}  and using \eqref{Nevenm+1-xmlemsameoreq}, we get
\beqs\begin{split}\dot{c}_i&=\gamma\left(-\sum_{j\neq i}\frac{\dot{x}_i-\dot{x}_j}{|x_i-x_j|^{2s+1}}-\delta''(t)\right)\\&
=-\gamma^2\sum_{j\neq i}|x_i-x_j|^{-2s-1}\left(\sum_{k\neq i}\frac{x_i-x_k}{2s|x_i-x_k|^{1+2s}}
-\sum_{l\neq j}\frac{x_j-x_l}{2s|x_j-x_l|^{1+2s}}\right)-\gamma \delta''(t)\\&
=O((1+t)^{-\frac{4s+1}{2s+1}}).
\end{split}\eeqs
Then 
\beq\label{errorestimeateIepcidot3} O(\ep^{2s+1}\dot{c}_j)=O(\ep^{2s+1}(1+t)^{-\frac{4s+1}{2s+1}}).\eeq
%Similarly
%\beq\label{errorestimeateIep23} O(\ep^{2s+1}\dot{c}_2)=O(\ep^{2s}\theta_\ep^{-4s}).\eeq

\noindent Next, from \eqref{uinfinity} and \eqref{x-xsjcase1promtheepbis}, we have for $j\neq i$
\beq\label{u2behavioinfty23}|\tilde{u}_j|\leq C\ep^{2s}|x-x_j|^{-2s}\leq C\ep^{2s}(1+t)^{-\frac{2s}{2s+1}}%(\theta_\ep^{-2s}+\ep^{\kappa-2s} \theta_\ep^{-k}+\ep^\alpha\theta_\ep^{-(1+2s)})
\eeq then using \eqref{c1c2behavior3}, we get for $j\neq i$

\beq\label{errorestimeateIep33}O(c_j\tilde{u}_j)=O(\ep^{2s}(1+t)^{-\frac{4s}{2s+1}})%+O[\ep^{2s}\theta_\ep^{-2s}(\ep^{\kappa-2s} \theta_\ep^{-k}+\ep^\alpha\theta_\ep^{-(1+2s)})]
,\eeq
and
\beq\label{errorestimeateIep43}O(\ep^{-2s}\tilde{u}_j^2)=O(\ep^{2s}(1+t)^{-\frac{4s}{2s+1}})%+O[\ep^{2s}(\ep^{\kappa-2s} \theta_\ep^{-k}+\ep^\alpha\theta_\ep^{-(1+2s)})^2]
.\eeq
Next, from \eqref{psi'infty} we know that for $|x|\geq\ep^{-1}C(1+t)^\frac{1}{1+2s}$ 
$$|\psi(x)|\leq  \left|\psi\left(\ep^{-1}C(1+t)^\frac{1}{1+2s}\right)\right|+C\ep^{2s}(1+t)^{-\frac{2s}{1+2s}}.$$
 Therefore, from \eqref{x-xsjcase1promtheepbis} and \eqref{c1c2behavior3} we get 
 \beq\label{errorestimeateIep53}O(c_j\psi_j)=O\left((1+t)^{-\frac{2s}{1+2s}}\psi\left(\ep^{-1}C(1+t)^\frac{1}{1+2s}\right)\right)+O(\ep^{2s}(1+t)^{-\frac{4s}{1+2s}}).
 \eeq
Let us choose $\delta_\ep$ such that 
\beq\label{thetaep3}\ep^\alpha,\, \ep^{2s},\,\psi(\ep^{-1}),\,
\ep^{\kappa-2s}=o(\delta_\ep)\quad\text{as }\ep\to0.\eeq 
 Then, from 
\eqref{iepsemifinal3}, \eqref{parenttermesti3}, \eqref{errorestimeateIep13}, \eqref{errorestimeateIepcidot3},   \eqref{errorestimeateIep33}, \eqref{errorestimeateIep43}, \eqref{errorestimeateIep53},  \eqref{thetaep3} and the definition of $\delta$ given in \eqref{wbarbarriercond2}, we obtain
\beq\label{Iepfinalbefordelta3}
\begin{split}I_\ep&=o(\delta_\ep)(1+t)^{-\frac{2s}{1+2s}}+\frac{1+2s}{\gamma}(\sigma_\ep+\delta_\ep)(1+t)^{-\frac{2s}{1+2s}}\\&
\geq o(\delta_\ep)(1+t)^{-\frac{2s}{1+2s}}+\frac{1+2s}{\gamma}\delta_\ep(1+t)^{-\frac{2s}{1+2s}}.\end{split}\eeq
\bigskip
being $\sigma_\ep\ge 0$.

%\noindent\emph{Case 2.}
%Suppose that  $x$ is close to $x_2(t)$ more than $\ep^  \alpha$:
%\beq\label{x-xs2leqeppow}|x-x_2(t)|\leq \ep^\alpha\quad\text{with }0<\alpha<\displaystyle\frac{\kappa-2s}{\kappa}.\eeq
%In this case, estimate \eqref{Iepfinalbefordelta} follows using the second equality in \eqref{ieplem2} and arguing as in Case 1.
%\bigskip

\noindent\emph{Case 2.} Suppose that for any $i=1,\ldots,N$ we have 
$$|x-x_i(t)|\geq \ep^\alpha.$$ 
If $x_i(t)$ is the closest particle to $x$, then  from \eqref{Nevenm+1-xmlemsameoreq}, for $j\neq i$, we have that
\beqs|x-x_j(t)|\geq C(1+t)^{1+2s}.\eeqs 
Then  estimates  \eqref{c1c2behavior3}, \eqref{errorestimeateIep13},  \eqref{errorestimeateIepcidot3},  \eqref{u2behavioinfty23}, \eqref{errorestimeateIep33}, \eqref{errorestimeateIep43} and \eqref{errorestimeateIep53} hold.
Moreover, using \eqref{uinfinity}, we have 
$$|\tilde{u}_i|\leq C\ep^{2s}|x-x_i|^{-2s}\leq C\ep^{2s(1-\alpha)},$$
%Moreover  estimate \eqref{u2behavioinfty3} holds for $\tilde{u}_i$ too 
and as a consequence, using in addition \eqref{u2behavioinfty23}, for $j\neq i$
$$O(\tilde{u}_i)(\ep^{-2s}\tilde{u}_j)=O(\ep^{2s(1-\alpha)}(1+t)^{-\frac{2s}{1+2s}})%\theta_\ep^{-4s})+O[\ep^{2s}(\ep^{\kappa-2s}\ \theta_\ep^{-k}+\ep^\alpha\theta_\ep^{-(1+2s)})^2]
.$$
Finally from   \eqref{c1c2behavior3}, we have 
 $$O(\tilde{u}_i)c_i=O(\ep^{2s(1-\alpha)}(1+t)^{-\frac{2s}{1+2s}})%+O[\ep^{2s}\theta_\ep^{-2s}(\ep^{\kappa-2s}\ \theta_\ep^{-k}+\ep^\alpha\theta_\ep^{-(1+2s)})]
 .$$
Then, if in addition to  \eqref{thetaep3},  we choose $\delta_\ep$ such that 
$$\ep^{2s(1-\alpha)}=o(\delta_\ep)\quad\text{as }\ep\to0,$$
%that $\theta_\ep$ satisfies  
%$$\ep^{2s(1-\alpha)}\theta_\ep^{-2s}=o(1)\quad\text{as }\ep\to0,$$
 from \eqref{ieplem3}, we obtain again \eqref{Iepfinalbefordelta3}.

Now, in both cases, from \eqref{Iepfinalbefordelta3}, for $\ep$ small enough we obtain that 
$$I_\ep\ge 0$$ and the proposition is proven.

\end{document}